\documentclass[reqno, 11pt]{amsart}
\usepackage{graphicx, enumerate, enumitem, amsmath, amssymb,amsthm,mathabx,manfnt,hyperref,braket}

\usepackage[OT2,T1]{fontenc}
\DeclareSymbolFont{cyrletters}{OT2}{wncyr}{m}{n}
\DeclareMathSymbol{\Sha}{\mathalpha}{cyrletters}{"58}

\newtheorem{thm}{Theorem}[section]
\newtheorem{cor}[thm]{Corollary}
\newtheorem{lem}[thm]{Lemma}
\newtheorem{prop}[thm]{Proposition}

\theoremstyle{definition}

\theoremstyle{remark}

\numberwithin{equation}{section}
\newcommand{\norm}[1]{\left\Vert#1\right\Vert}
\newcommand{\abs}[1]{\left\vert#1\right\vert}
\newcommand{\pair}[2]{\left\langle#1,#2\right\rangle}
\newcommand{\Dd}{\mathcal{D}}
\newcommand{\Ee}{\mathcal{E}}

\newcommand{\rl}{{\mathbb{R}}}
\newcommand{\cx}{{\mathbb{C}}}
\newcommand{\id}{{\mathsf{id}}}
\newcommand{\Mm}{\mathfrak{M}}
\newcommand{\Nn}{{\mathfrak{N}}}
\newcommand{\Xx}{\mathsf{X}}

\newcommand{\dbar}{\overline{\partial}}

\newcommand{\tensor}{\otimes}
\newcommand{\csor}{ \widehat{\otimes} }

\newcommand{\bc}{{\mathsf{bc}}}
\newcommand{\ce}{{\mathsf{ce}}}
\newcommand{\dist}{{\rm dist}}
\newcommand{\supp}{{\rm supp}}
\title[Distributional Boundary Values]{Distributional boundary values of holomorphic functions on product domains}\thanks{This work was partially supported by a grant from the Simons Foundation (\#316632 to Debraj Chakrabarti). Debraj Chakrabarti was also partially supported by an Early Career internal grant from Central Michigan University. Rasul Shafikov was partially supported by an NSERC grant.}
\author{Debraj Chakrabarti}
\address{Department of Mathematics, Central Michigan University, Mt. Pleasant,  MI 48859,  USA}
\email{chakr2d@cmich.edu}
\author{Rasul Shafikov}
\address{Department of Mathematics, University of Western Ontario,  London, Ontario, Canada, N6A 5B7 }
\email{shafikov@uwo.ca}
\begin{document}
\maketitle
\section{Introduction}
\subsection{Distributional Boundary Values} The study of boundary values of holomorphic functions as generalized functions has a long history,  going back to 
the theory of analytic functionals developed by Fantappi\`{e} in the 1920's and 30's. The development 
of the theory of topological vector spaces and distributions led to significant progress in this problem in the 1950's and 60's 
by K{\"o}the, Silva, Grothendieck, Sato, Martineau, Tillmann and  many others. See the book review \cite{horvath-review} for a
short history of the topic, and  \cite{martineau1} for an annotated bibliography of the early contributions till 1964.  Such generalized boundary values may be studied either as distributions in the sense of Sobolev and Schwartz 
or as hyperfunctions in the sense of Sato. The former approach, which is adopted in this paper,
allows $\mathcal{C}^\infty$-smooth boundaries, but requires the 
holomorphic functions  to grow at most polynomially as one approaches the boundary. If we use hyperfunctions, we can obtain boundary values of all holomorphic functions, but we must restrict the boundary to be 
real analytic (see \cite{polkingwells}). We note however that many of the considerations of this paper apply 
to hyperfunction boundary values as well, and this aspect will be discussed in detail in a forthcoming paper. It is also possible
to obtain boundary values of other holomorphic objects ($p$-forms, sections of vector bundles) by routine extensions of 
the methods of this paper.

The motivation
of this paper is to try to generalize the notion 
of distributional boundary values to holomorphic functions defined on domains which have {\em piecewise smooth} boundaries.  Recall that there exist distributional boundary values 
of holomorphic functions of polynomial growth defined on a ``wedge'' attached to a generic ``edge'' (see \cite{baouendibook,cordarotreves}).  This suggests that we might be
able to define boundary values of holomorphic functions on a piecewise smooth domain, when the corners of the domain are generic, i.e., the corners are CR manifolds. This turns out to be correct, and we can define boundary values of holomorphic functions of polynomial 
growth on the class of  {\em piecewise smooth domains with generic corners.} The boundary value is realized as a de Rham current in the ambient complex manifold, a well-known formalism in complex analysis (see \cite{hala1}). 

The main thrust of this paper is to study global holomorphic extension properties of ``CR'' boundary currents on piecewise smooth domains, generalizing the classical Bochner-Hartogs theorem.  We give a complete solution of this problem for products of smoothly bounded domains. This may be considered a generalization of one of the earliest results in several complex variables due to Hartogs, the extension of a holomorphic function from a neighborhood of the boundary of a polydisc to the whole polydisc (see \cite[Theorem~1, page 12] {narasimhan}, and also \cite{landucci}). The invariant nature 
of the de Rham currents allows us  to state our results for domains in general complex manifolds, without any cohomological constraints. 

\subsection{Main Results}
Let $\Omega$ be a relatively compact  domain in a complex  manifold  $\Mm$ that may be written as 
\begin{equation}\label{eq-omegaint}
 \Omega = \bigcap_{j=1}^N \Omega_j,
\end{equation}
where each $\Omega_j\subset\Mm$ is a  smoothly bounded domain.  If  for each subset $S\subset\{1,\dots,N\}$ the intersection $B_S=\bigcap_{j\in S} b\Omega_j$, if non-empty,
is a CR manifold of CR-dimension $n-\abs{S}$, we say that  $\Omega$ is  a {\em domain with generic corners.}
Domains with generic corners are significant in many areas of complex analysis, see \cite{webster82, forst93,barrettduality, chak-kaushal2}. The most important examples of  domains with generic corners  are the {\em  product domains. } 

Impose on $\Mm$ any metric compatible with its topology (we assume that all manifolds appearing in this paper are countable at infinity). Throughout the paper we denote by ${\rm dist}(z,X)$ the distance from a point $z\in\Mm$ to a set $X$ induced by the chosen metric.
If $\Omega\Subset\Mm$ is a relatively compact domain,
then a holomorphic $f\in \mathcal{O}(\Omega)$ is said to be of {\em polynomial growth} if there is a $C>0$ and $k\geq 0$ such that we have for each $z\in \Omega$ that
\[ \abs{f(z)}\leq \frac{C}{\dist(z,\partial\Omega)^k}.\]
We denote the space of holomorphic functions of polynomial growth on $\Omega$ by $\mathcal{A}^{-\infty}(\Omega)$. This space has a natural topology (cf. Section~\ref{sec-ainfty} below). We can show that holomorphic functions of polynomial growth on domains with generic corners have boundary values in the sense of distributions:

\begin{thm}\label{thm-bcexistence} Let $\Omega$ be a domain with generic corners in a complex manifold $\Mm$, and let $f\in \mathcal{A}^{-\infty}(\Omega)$. There is a $(0,1)$-current $\bc f \in \Dd'_{0,1}(\Mm)$ such that the following 
holds. If $U$ is a coordinate neighborhood of $\Mm$, and $\psi\in\Dd^{n,n-1}(\Mm)$  is a smooth $(n,n-1)$ form which has support
in $U$, and there is a vector $v\in \cx^n$ such that in the coordinates on $U$, the vector $v$ points outward from $\Omega$ along each $\partial\Omega_j$  inside $U$,
 then we have
 \begin{equation}\label{eq-bc}
 \pair{\bc f}{\psi}= \lim_{\epsilon\downarrow 0}\int_{\partial\Omega} f_\epsilon \psi ,
\end{equation}
where   $f_\epsilon(z)= f(z-\epsilon v)$.
\end{thm}

We refer to $\bc f$ as the {\em boundary current} induced by the holomorphic function $f$ of polynomial growth. The question naturally arises of characterizing the range of the map $\bc$, i.e., describing which currents $\gamma\in \Dd'_{0,1}(\Mm)$ arise as boundary values of 
holomorphic functions of polynomial growth on $\Omega\Subset\Mm$, where $\Omega$ is a domain with generic corners in $\Mm$. There is a simple 
necessary condition on all boundary currents which we first note.  We say that a $(0,1)$-current $\gamma\in \Dd'_{0,1}(\Mm)$ on a complex manifold  $\Mm$   satisfies the  {\it Weinstock orthogonality condition} (cf. \cite{weinstock1})
with respect to a domain $\Omega\subset \Mm$, or simply the Weinstock condition, if for $\omega\in \Dd^{n,n-1}(\Mm)$,
\begin{equation}\label{eq-weinstock}
\dbar\omega=0 {\rm\ on\ } \overline{\Omega} \ \ \Longrightarrow\ \ \pair{\gamma}{\omega}=0.
\end{equation}
This is a generalization of the usual tangential Cauchy-Riemann equations 
for the boundary values of holomorphic functions, in fact, for domains in $\mathbb C^n$ with connected complement, the Weinstock condition is equivalent to $\gamma$ being $\overline\partial$-closed (see the proof of Corollary~\ref{cor-bochnerhartogs}  below). In Proposition~\ref{prop-weinstock} we show that the Weinstock orthogonality is a necessary condition for $\gamma$ to be the boundary value of a function $f\in \mathcal{A}^{-\infty}(\Omega)$. 
 
Now we consider the case of smoothly bounded domains. If $\Omega\Subset \Mm$ is a smoothly bounded domain,
we define a subspace $\mathcal{X}^{0,1}_\Omega(\Mm)$ of $\Dd'_{0,1}(\Mm)$ as follows. A current 
$\gamma\in \Dd'_{0,1}(\Mm)$ belongs to $\mathcal{X}^{0,1}_\Omega(\Mm)$ if and only if $\gamma$ satisfies 
the following two conditions:
\begin{enumerate}\item 
$\gamma$ satisfies the Weinstock condition with respect to $\Omega$.
\item There is a {\it face distribution} $\alpha\in \Dd'_0(\partial\Omega)$ which induces $\gamma$ in the following way: if $\iota:\partial\Omega\to \Mm$ is the inclusion map, we have 
\begin{equation}\label{eq-facedistn}
\gamma=\iota_*(\alpha)^{0,1}.
\end{equation}
\end{enumerate}
Here $\iota_*$ is the pushforward operation on currents by $\iota$ (see \eqref{eq-pushforward} below) and for a 1-current $\theta$ on a complex manifold, we write  \[ \theta=\theta^{0,1}+\theta^{1,0},\] the decomposition of $\theta$ into parts of bidegree $(0,1)$ and $(1,0)$. The space $\mathcal{A}^{-\infty}(\Omega)$ 
has a natural topology, which is defined formally in Section~\ref{sec-ainfty}.  Further $\mathcal{X}^{0,1}_\Omega(\Mm)$ is a closed subspace of the space $\Dd'_{0,1}(\Mm)$, and therefore carries 
the subspace topology. We have the following:

\begin{thm}\label{thm-smoothcase} 
Let $\Omega\Subset\Mm$ be a domain with $\mathcal{C}^\infty$-smooth boundary. Then the map
\[ \bc:\mathcal{A}^{-\infty}(\Omega)\to \mathcal{X}^{0,1}_\Omega(\Mm)\]
is an isomorphism of topological vector spaces.
\end{thm}

When $\Mm=\cx^n$, Theorem~\ref{thm-smoothcase}, allows us to obtain the distributional 
version of the Bochner-Hartogs phenomenon, thus recapturing  a result of Straube 
(see \cite[Thm~2.2]{straube84}).

\begin{cor}\label{cor-bochnerhartogs} 
Let $n\geq 2$, and let $\Omega\Subset\cx^n$ be a smoothly bounded domain such that 
$\cx^n\setminus\overline{\Omega}$ is connected. Suppose that $\gamma\in \Dd^{n,n-1}(\mathbb C^n)$
is $\overline\partial$-closed and~\eqref{eq-facedistn} holds. Then there is a holomorphic 
$f$ on $\Omega$ of polynomial growth, such that $\gamma=\bc f$.
\end{cor}

See \cite{range} for the history of the global holomorphic extension theorem for CR functions. Distributional analogs of other classical global holomorphic 
extension results (cf. \cite{kohnrossi}) can also be deduced from Theorem~\ref{thm-smoothcase}, using 
approximation properties of forms in particular manifolds.

While the problem of identifying boundary currents for an arbitrary domains with generic corners remains open, 
we are able to give a complete characterization of the range of the operator $\bc$ in the case when $\Omega$ is 
a {\em product domain}.  Let $\Mm_1,\dots,\Mm_N$ be complex manifolds, and let 
\begin{equation}\label{eq-mproduct}
\Mm=\Mm_1\times\dots\times\Mm_N
\end{equation}
 be their product as a complex manifold. For $j=1,\dots, N$ let $D_j\Subset\Mm_j$ be be a domain with $\mathcal{C}^\infty$-smooth boundary. By a product domain, we mean a domain  $\Omega$ of the form 
\begin{equation}\label{eq-omegaproduct}
\Omega=D_1\times\dots \times D_N.
\end{equation}
 For $j=1,\dots, N$, we define  a smoothly bounded domain  $\Omega_j\subset \Mm$ by setting
\begin{equation}\label{eq-omegaj}
\Omega_j = \Mm_1\times\dots \times D_j \times\dots\times\Mm_N,
\end{equation}
where  the $j$-th factor is $D_j$ and all other  factors are $\Mm_k$ . Then we can represent the domain $\Omega$ as an intersection \eqref{eq-omegaint}, and it is easy to see that each corner is a CR manifold, so product domains 
have generic corners. 

With $\Omega$ and $\Mm$ as above, we introduce a subspace $\mathcal{Y}^{0,1}_\Omega(\Mm)$ of
$\Dd'_{0,1}(\Mm)$. A current $\gamma\in \Dd'_{0,1}(\Mm)$ belongs to $\mathcal{Y}^{0,1}_\Omega(\Mm)$ if the 
following conditions are satisfied:
\begin{enumerate}
	\item $\gamma$ satisfies the Weinstock condition \eqref{eq-weinstock} with respect to $\Omega$.
	
	\item Suppose that the piecewise smooth domain $\Omega$ is represented as an intersection of smoothly bounded domains as in \eqref{eq-omegaint}. For each 
$j=1,\dots, N$, let 
\begin{equation}\label{eq-iotaj}
\iota^j:\partial\Omega_j\to \Mm
\end{equation}
 denote the inclusion map. There are distributions $\alpha_j\in \Dd'_0(\partial\Omega_j)$ with support in $\partial\Omega_j\cap \overline{\Omega}$ such that
we can write
\begin{equation}\label{eq-facewise}
\gamma=\sum_{j=1}^N\left(\iota^j_*(\alpha_j)\right)^{0,1}.
\end{equation}
We will call the distributions $\alpha_1,\dots, \alpha_N$ the {\em face distributions} associated with the current $\gamma$.

	\item The third condition, which we call {\it canonicality of face distributions} is somewhat technical, and will be fully explained below in Section~\ref{sec-y01}. Informally, it can be understood as follows. Given a function 
$f\in \mathcal{A}^{-\infty}(\Omega)$, there exists the extension of $f$ as a distribution in $\mathcal{D}'_0(\Mm)$
with the property that it vanishes outside $\overline \Omega$ and its values on $\partial \Omega$ are 
determined in a limit process from the values in $\Omega$, see Theorem~\ref{thm-ce}. This will be called the 
{\it canonical extension} of $f$. A similar canonical extension exists for the distributions 
$\alpha_j\in \Dd'_0(\partial\Omega_j)$ defined by~\eqref{eq-facewise}. The condition now
is that the canonical extensions of $\alpha_j$ agree with $\alpha_j$, see \eqref{eq-canonicality} below for the exact statement. In particular, this condition ensures that one can talk about boundary values of the face distributions themselves along higher codimensional strata. 
\end{enumerate}
We note that all three conditions above are satisfied by boundary currents of holomorphic functions. In fact,
we have the following characterization of the distributional boundary values of holomorphic functions on product domains:

\begin{thm}\label{thm-product} Let $\Omega$ be a product domain as above. Then for each $f\in \mathcal{A}^{-\infty}(\Omega)$, we have $\bc f\in \mathcal{Y}^{0,1}_\Omega(\Mm)$, and 
the map
\[ \bc: \mathcal{A}^{-\infty}(\Omega)\to \mathcal{Y}^{0,1}_\Omega(\Mm)\]
is an isomorphism of topological vector spaces.
\end{thm}

In the last section, we relate this notion of boundary value with the more usual notion of boundary value on the \v{S}ilov 
boundary. 
 
\bigskip

\noindent{\bf Acknowledgments.} We thank Christine Laurent-Thi\'{e}baut and David Barrett for sharing with us many of the ideas on which this paper is based. We also thank Sagun Chanillo,  Evgeny Poletsky,  Jean-Pierre Rosay,  Mei-Chi Shaw and Emil Straube for interesting discussions on the topic of this paper.

\section{Existence and basic properties of boundary currents}

\subsection{The space  $\mathcal{A}^{-\infty}(\Omega, \Xx)$}\label{sec-ainfty}
We consider holomorphic functions and distributions with values in a  Banach space, see \cite{trevesbook} 
for basic facts on Banach-valued holomorphic functions and distributions. This will be needed in Proposition~\ref{prop-extension}, which will be later used in the proof of Proposition~\ref{prop-alpha1alpha2}.
Let $\Mm$ be a complex manifold, and let $\Omega\Subset\Mm$ be a relatively compact domain. We endow the
manifold $\Mm$ with an arbitrary Riemannian metric. All distances are measured with respect to this metric.  Let $\Xx$ be 
a Banach space, and denote by $\mathcal{O}(\Omega, \Xx)$ the space of all 
$\Xx$-valued holomorphic functions on $\Omega$. We say that a 
function $f\in \mathcal{O}(\Omega, \Xx)$ is of {\em polynomial growth} if there is a non-negative integer $k$ and a $C>0$ such that 
 \begin{equation}\label{eq-slowgrowth} \norm{f(z)}_\Xx\leq \frac{C}{\dist(z,b\Omega)^k}.\end{equation}
For a fixed $k$, we denote by $\mathcal{A}^{-k}(\Omega, \Xx)$ the space of $\Xx$-valued holomorphic functions on $\Omega$ which satisfy the estimate 
\eqref{eq-slowgrowth}. Then $\mathcal{A}^{-k}(\Omega, \Xx)$ is a Banach space with the norm 
\[ \norm{f}_{\mathcal{A}^{-k}} = \sup_{z\in\Omega} \left\{ \norm{f(z)}_\Xx \dist(z,b\Omega)^k\right\}.\]
We denote the space of all $\Xx$-valued holomorphic functions of polynomial growth on $\Omega$ by $\mathcal{A}^{-\infty}(\Omega,\Xx)$:
\begin{equation}\label{eq-ainfty}
  \mathcal{A}^{-\infty}(\Omega,\Xx)= \bigcup_{k=0}^\infty \mathcal{A}^{-k}(\Omega,\Xx),
\end{equation}
and endow $\mathcal{A}^{-\infty}(\Omega,\Xx)$ with the inductive limit topology.

\subsection{Notation for currents} For  de Rham currents  we will use the following standard notation and terminology. If $\Mm$ is a differentiable manifold of dimension $N$, we denote by $\Dd^q(\mathfrak{M})$  the space of smooth compactly supported $q$-forms on $\mathfrak{M}$, which is a topological vector space with the standard inductive limit topology. We denote by $\Dd'_q(\Mm)$  (space of currents of degree $q$, or $q$-currents) the topological
 dual of the space $\Dd^{N-q}(\mathfrak{M})$,  endowed with the strong topology (the topology of uniform convergence on bounded subsets of $\Dd^{N-q}(\mathfrak{M})$, see \cite[p.~198ff]{trevesbook}.  A {\em distribution} is a $0$-current, and given a locally integrable function $u$ on a manifold $\Mm$, we identify
 $u$ with the distribution (i.e. 0-current)  $\phi\mapsto \int_\Mm \phi$, where $\phi\in \Dd^{\dim_{\rl}\Mm}(\Mm)$ is a compactly supported form of top degree on $\Mm$. If $\Xx$ is a Banach space, then an {\em $\Xx$-valued
distribution} is an element of the space $\Dd'_0(\Mm,\Xx)$ of continuous linear maps from $\Dd^N(\Mm)$ to $\Xx$. 
 
 When $\Mm$ is a complex manifold of complex dimension $n$, we let $\Dd^{p,q}(\mathfrak{M})$ be the space of smooth compactly supported $(p,q)$-forms, and 
 $\Dd'_{p,q}(\Mm)$  the space of $(p,q)$-currents, i.e., the dual (with the strong  topology) of  $\Dd^{n-p,n-q}(\mathfrak{M})$.  
If $f:\Mm\to \Nn$ is a mapping of smooth manifolds, and $\gamma$ is a current on $\Mm$, we denote by $f_*\gamma$ the {\em pushforward} of the 
current $\gamma$ by the map $f$. Recall that
\begin{equation}\label{eq-pushforward}
\pair{f_* \gamma}{\phi} = \pair{\gamma}{f^*\phi},
\end{equation}
where $\phi$ is a smooth compactly supported form of appropriate degree and $f^*$ denotes the pullback operator on forms.

\subsection{Distributional extensions}We now consider the problem of extending a holomorphic function 
defined on the domain $\Omega$ to a distribution defined on $\Mm$. A 0-current $F\in \Dd'_{0,0}(\Mm,\Xx)$ 
 will be called an  { \em ($\Xx$-valued) distributional extension} of $f$ if $F|_\Omega=f$.

Since $\Xx$ is a Banach space, any $\Xx$-valued current $\gamma$ is locally of finite order. This means that, after choosing an arbitrary Riemannian metric on $\Mm$, for each compact $K\subset \Mm$, there is a $C>0$ and an integer $k\geq 0$ (the local order
of $\gamma$ on $K$) such that for any test form $\phi$ of appropriate degree with with support in $K$ we have
\[ \norm{\langle \gamma,\phi\rangle}_{\Xx}\leq C \norm{\phi}_{\mathcal{C}^k(\Mm)},\]
where the $\mathcal{C}^k$-norm is defined with respect to the Riemannian metric on $\Mm$.
In particular, it follows that the distributional 
extension $F$, being compactly supported is of finite order on the whole of $\Mm$.

\begin{prop}\label{prop-extension}
If $f\in \mathcal{O}(\Omega,\Xx)$ admits a distributional extension $F\in \Dd'_{0,0}(\Mm,\Xx)$, 
then $f\in \mathcal{A}^{-\infty}(\Omega,\Xx)$.
\end{prop}

\begin{proof} After using a system of local holomorphic coordinates centered at a boundary point, it is sufficient to prove the 
result when $\Mm=\cx^n$.  We use a classic argument of Bell (\cite[Lemma~2]{bell_harm}.)  Let $\chi\in\mathcal{D}(\cx^n)$ be a compactly supported smooth radial  function with support in the unit ball   such that
$ \int \chi dV =1,$
and for a fixed $z$ in $\Omega$ define the form $\phi\in \Dd^{n,n}(\Mm)$ by
\[ \phi(w)= \left(\frac{{\rm dist}(z,\partial\Omega)}{2}\right)^{-2n}\cdot \chi\left( \frac{w-z}{\frac{1}{2}{\rm dist}(z,\partial\Omega)}\right) d{\rm vol},\]
where $d{\rm vol}$ is the standard volume form of $\cx^n$.
Then $\int\phi =1$ as well, and 
$\phi$ is radially symmetric about $z$. A direct computation reveals that there is a constant $C_0$  depending only on the function $\chi$ and independent of $z\in \Omega$ such that
\begin{equation}\label{eq-phinorm} \norm{\phi}_{\mathcal{C}^k(\cx^n)} \leq \frac{C_0}{{\rm dist}(z,\partial\Omega)^{k+2n}}.\end{equation}
Assume that $F$ is of order  $k$. Then, by the mean value theorem,
\begin{align*}
\norm{f(z)}_\Xx &= \norm{\int_\Omega f \phi }_\Xx 
= \norm{\langle F,\phi \rangle}_\Xx 
\leq C\norm{\phi}_{\mathcal{C}^k(\cx^n)} 
\leq \frac{C}{{\rm dist}(z,\partial\Omega)^{k+2n}}.
\end{align*}\end{proof}

\subsection{ $\mathcal{A}^{-\infty}(\Omega,\Xx)$ as a DFS space}\label{sec-dfs} Recall that a {\em DFS space} is a topological vector space isomorphic to the strong dual of a Fr\'{e}chet-Schwartz space.  We note the following two facts:
\begin{prop} \label{prop-ainftydfs} If $\Omega$ is a relatively compact domain in $\Mm$, then $\mathcal{A}^{-\infty}(\Omega,\Xx)$ is a 
DFS space for any Banach space $\Xx$.
\end{prop}
\begin{proof} By a result in functional analysis (\cite[Proposition~25.20]{vogt} or \cite[Appendix A, Sections 5 and 6]{morimoto}), a DFS space may be characterized as the inductive limit $E={\rm ind}_n E_n$ of an increasing sequence of Banach spaces $\{E_n\}_{n\in \mathbb{N}}$
 with the property that for each $n$, there is an $m>n$ such that the embedding $E_n\to E_m$ is compact. Therefore, it suffices to show that the inclusion map $\mathcal{A}^{-k}(\Omega,\Xx)\hookrightarrow \mathcal{A}^{-(k+1)}(\Omega,\Xx)$ is a compact linear map of Banach spaces.	Let $\{f_\nu\}$  be a sequence in the unit ball of $\mathcal{A}^{-k}(\Omega,\Xx)$. By \eqref{eq-slowgrowth}, on each compact $K\subset \Omega$, the family $\{f_\nu\}$ is uniformly bounded and therefore, by a Banach-valued version of Montel's theorem, there exists $f\in \mathcal{O}(\Omega,\Xx)$ such that a subsequence $f_{\nu_j}$ converges to $f$ uniformly on  compact subsets of $\Omega$. Note that $f_{\nu_j}$ need not to converge to $f$ in $\mathcal{A}^{-k}(\Omega,\Xx)$.
 The estimate \eqref{eq-slowgrowth} implies that the limit $f$ lies in the closed unit ball of $\mathcal{A}^{-k}(\Omega,\Xx)$, so that
 $\norm{f-f_{\nu_j}}_{\mathcal{A}^{-k}}\leq 2$ for each $j$. Let $g_j(z)=\norm{f(z)-f_{\nu_j}(z)}_{\Xx}\dist(z,\partial\Omega)^{k+1}$, and let $\epsilon>0$. On the open set $\{z\in \Omega\colon \dist(z,\partial\Omega)< \frac{\epsilon}{2}\}$, we see that $g_j(z) \leq \norm{f-f_{\nu_j}}_{\mathcal{A}^{-k}}\dist(z,\partial\Omega)< \epsilon$. On the complementary compact set $\{z\in\Omega \colon \dist(z,\partial\Omega)\geq \frac{\epsilon}{2}\}$,
 as $j\to \infty$, we have $f_{\nu_j}\to f$ uniformly, so that we can find an $N_\epsilon$ so large that if $j> N_\epsilon$, 
 then $\norm{f(z)-f_{\nu_j}(z)}_{\Xx}<\frac{\epsilon}{\delta^{k+1}}$, where 
 $\delta= \max_{z\in\Omega}\dist(z,\partial\Omega)$. 
 Then for $j> N_\epsilon$ and for each $z\in \Omega$, we have $g_j(z)< \epsilon$. It follows that $f_{\nu_j}\to f$ in $\mathcal{A}^{-(k+1)}(\Omega, \Xx)$ and therefore, 
 the map $\mathcal{A}^{-k}(\Omega,\Xx)\hookrightarrow \mathcal{A}^{-(k+1)}(\Omega, \Xx)$ is 
 compact.
\end{proof}

We will also use the following fact, which is a consequence in the usual way of the closed graph theorem 
for DFS spaces, a proof of which can be found in \cite[Appedix A, Corollary A.6.4]{morimoto}.

\begin{prop}\label{res-dfsinverse} 
If   $E,F$ are DFS spaces, and $u:E\to F$ is a continuous linear map which is a  set-theoretic bijection, 
then $u$ is an isomorphism of topological vector spaces.
\end{prop}

\subsection{Canonical Extension}
We now show that holomorphic functions of polynomial growth on domains with generic corners admit distributional extensions,
by constructing one such extension. This extension will be called the {\em canonical extension} of the holomorphic function.
Our proof will use a method of Barrett (see\cite{barrettduality}). The notion of canonical extension is due to Andr\'{e} Martineau, who noted its existence for smoothly bounded planar domains in \cite{martineau1}.

By a {\em face} of the domain $\Omega$ in \eqref{eq-omegaint} we mean the subset $\partial\Omega_j\cap \partial\Omega$ of the boundary of $\Omega$ for some~$j$.

\begin{thm}\label{thm-ce} 
Let $\Mm$ be a complex manifold of complex dimension $n$, and let $\Omega\Subset\Mm$ be a domain with 
generic corners. There is a unique continuous linear map 
\[\ce: \mathcal{A}^{-\infty}(\Omega)\to \mathcal{D}'_0(\Mm),\]
such that the following conditions hold.
\begin{enumerate}
\item $\ce f$ is a distributional extension of the function $f$, i.e., $(\ce f)|_\Omega=f$.
 
\item   If $U$ is a coordinate neighborhood of $\Mm$, and $\phi\in\Dd^{n,n}(\Mm)$ has support
in $U$, and there is a vector $v\in \cx^n$ such that in the coordinates on $U$, the vector $v$ points 
outward from $\Omega$ along each $\partial\Omega_j$  meeting the support of $\phi$, then we have
\begin{equation}\label{eq-ce}
 \pair{\ce f}{\phi} = \lim_{\epsilon\downarrow 0}\int_\Omega f_\epsilon \phi ,
\end{equation}
where   $f_\epsilon(z)= f(z-\epsilon v)$.
\item In the special case when $f$ is continuous on $\overline{\Omega}$, the distribution $\ce f$ is induced by the function 
which coincides with $f$ on $\overline\Omega$ and vanishes outside $\overline{\Omega}$.
\end{enumerate}
\end{thm}

 \begin{proof}
We may cover the boundary $b\Omega$ by a finite collection $\{U_j\}_{j=1}^M$ of coordinate charts,
such that in the local coordinates of each chart $U_j$, there is a vector $v_j$ of the type referred to above.
Thus, without loss of generality we may assume that $\Mm=\cx^n$, and the existence of $\ce$ would 
follow provided we can show the existence of the limit \eqref{eq-ce}, and its independence of the choice 
of the vector $v\in\cx^n$. In the heart of our proof is the integration by parts argument due to Barrett 
(see \cite{barrettduality}). 

To illustrate the idea, consider first the simplest case when $\Omega$ has smooth boundary,  
and let 
with $r$ be  a defining function of $\Omega$ (see \cite{straube84}).  Let $T$ be a vector field of type
$(0,1)$ such that $Tr\equiv 1$ near $\partial\Omega$.  Let $U$
be a neighborhood of a point in $\partial\Omega$, and let $v$ a vector such that $U$ and $v$ satisfy 
the conditions in (2) of the theorem. By shrinking $U$, we may assume that $Tr\equiv 1$ on $U$.
For a $(n,n)$-form $\phi$ supported
in $U$ we write $\phi = \phi_0 dV$, where $dV$ is the standard volume form of $\cx^n$. Let $\epsilon>0$, and let $f_\epsilon$ be as in \eqref{eq-ce}.  Then for
any positive integer $s$, by applying the integration by parts formula $s$ times we obtain that
\[
\int_\Omega f_\epsilon \phi = \frac{1}{s!}\int_\Omega (r^s f_\epsilon )\, ((T^*)^s\phi_0) dV,
\]
where $T^*= -(T+{\rm div}\, T)$ is the formal transpose of the vector field $T$, and we use  the fact that $T(r^kf_\epsilon)= k r^{k-1}f_\epsilon$ on $U$ for each integer $k\geq 0$. Note that the boundary terms 
vanish at each step since $f_\epsilon$ is smooth up to $\partial\Omega$.
Using that $f$ is of polynomial growth, we choose $s$ such that the function $r^s f$ is continuous on 
$\overline \Omega$. Then as $\epsilon\to 0$, we have $r^s f_\epsilon\to r^s f$ {\em uniformly} on $\overline{\Omega}$, so that we obtain by letting $\epsilon\to 0$ that 
\begin{equation}\label{eq-cesmooth}
\lim_{\epsilon\downarrow 0}\int_\Omega f_\epsilon \phi = \frac{1}{s!}\int_\Omega (r^s f )\, ((T^*)^s\phi_0) dV.
\end{equation}
Therefore, the limit on the left hand side exists and is given by the expression on the right hand side, which does not involve 
$v$. Hence, the limit is independent of $v$, and also of $s$ (as long as $r^s f$ is continuous on $\overline{\Omega}$), since the left hand side is independent  of $s$. We conclude that $\pair{\ce f}{\phi}$ exists for $\phi$ supported in $U$ and is given by the expression on the right hand side of \eqref{eq-cesmooth}, and the global existence of $\ce f$ follows from a partition of unity argument. 

In fact, for a general domain with generic corners  in $\mathbb C^n$ it is possible to give a formula analogous to the right-hand side of \eqref{eq-cesmooth} for computing $\ce f$ directly. To state this formula (equation (2.1a) of \cite{barrettduality}), we use the following notation. Let $r_k$ denote a defining function of the domain $\Omega_k$  in the representation \eqref{eq-omegaint}. Note that the condition that the corners of a domain with generic corners are generic CR manifolds is  equivalent to the 
following: at each point  in the intersection $\bigcap_{j\in S}\partial\Omega_j$, we have  $\bigwedge_{j\in S}\dbar r_j\not=0$. Therefore, in a neighborhood of each point $p\in \mathbb C^n$,  we can find $N$ vector fields $T_1^{(p)},\dots, T_N^{(p)}$ of type $(0,1)$ such that $T_j^{(p)}r_k=\delta_{jk}$ in a neighborhood of $p$ whenever $r_j(p)=r_k(p)=0$. By a partition of unity argument, we obtain vector 
fields $T_j, j=1,\dots, N$, on $\cx^n$ of type $(0,1)$ such that $T_jr_k=\delta_{jk}$ on a neighborhood  $U_{jk}$ of $\partial\Omega_j\cap\partial\Omega_k$. Let $T_j^*$ denote the first order differential 
operator on $\Mm$ which is the formal transpose of the vector field  $T_j$ with respect to the standard 
bilinear pairing $(u,v)\mapsto \int_\Mm uv dV$, i.e., for smooth compactly supported $u,v$, we have 
$\int_\Mm(T_j u)v dV= \int_\Mm u(T_j^*v) dV$.

For a subset $S\subset\{1,2,\dots, n\}$, let  $U_S = \bigcap_{j,k\in S} U_{jk}\setminus \bigcup_{\ell \not\in S} b\Omega_\ell.$ In particular, $U_{\emptyset} = \cx^n \setminus \bigcup_{\ell=1}^N b\Omega_\ell$. Then the family $\{U_S\}$, as $S$ runs over all possible subsets 
of $\{1, 2, \ldots, N\}$ including the empty set, is an open cover of $\cx^n$. 
Let $\{\chi_S\}$ be a partition of unity subordinate to this cover.  We now define, for an $f\in \mathcal{A}^{-\infty}(\Omega)$ and a $\phi=\phi_0 dV$, the distribution $\ce f$ by the prescription
\begin{equation}\label{eq-cedef}
  \pair{\ce f}{\phi} = \sum_{S\subset \{1,\dots, N\}}\int_\Omega \left( \prod_{j\in S} \frac{r_j^{s_j}}{s_j!}f\right) \left(\prod_{j\in S}(T_j^*)^{s_j}
\right)\left(\chi_S \phi_0\right) dV,
\end{equation}
where $(s_1,\dots,s_N)\in \mathbb{N}^N$ is such that $r_1^{s_1}\dots r_N^{s_N}f$ is continuous on 
$\overline{\Omega}$.

Repeated integrations by parts,
using the relations $T_jr_k=\delta_{jk}$ and $T_j f =0$ (since $f$ is holomorphic) shows that when $f$ 
bounded, we have $\pair{\ce f}{\phi}= \int_\Omega f \phi$, i.e., $\ce f = f [\Omega]$,
which shows, in particular, that at least in this case, $\ce f$ is defined independently of the choice of the tuple
 $(s_1,\dots, s_N)$ (see  \cite{barrettduality,chak-kaushal2} for details.)  The general argument is similar to that in the smooth case above. Assuming $\phi$ and $f_\epsilon$ are as in \eqref{eq-ce}, the same integration by parts argument shows that limit in \eqref{eq-ce} is given by \eqref{eq-cedef} ,  which shows that $\ce f$ is defined by 
\eqref{eq-cedef} independently of $(s_1,\dots,s_N)$, and that limit on the right-hand side of \eqref{eq-ce} is independent of 
the particular vector $v$ used to define  $f_\epsilon$.

To show the uniqueness, assume that there exists another distribution, say, $h\in \mathcal{D}'_0(\Mm)$ that
satisfies conditions (1) and (2) of the theorem for a given function $f$. Then it follows from \eqref{eq-ce} 
that in a coordinate neighborhood $U$ of any point $\partial\Omega$ on which $v$ exists, we have 
$h|_U = {\ce f}|_U$. This implies the uniqueness. 

The continuity of $\ce:\mathcal{A}^{-\infty}(\Omega)\to \Dd'_0(\Mm)$ follows from the 
expression~\eqref{eq-cedef}. 

Finally, the property (3) also follows from the representation \eqref{eq-cedef}.  Then we have $s_1=s_2=\dots=s_N=0$,
and $\pair{\ce f}{\phi}= \int_\Omega f\phi_0dV$.
\end{proof}

\subsection{Proof of Theorem~\ref{thm-bcexistence}} We define an operator $\bc:\mathcal{A}^{-\infty}(\Omega)\to \Dd'_{0,1}(\Mm)$ by setting
\begin{equation}\label{eq-bcdef}
\bc = -\dbar \circ \ce.
\end{equation}
Since $\ce$ and $\dbar$ are continuous, so is $\bc$. Let $f\in \mathcal{A}^{-\infty}(\Omega)$. Then we have, with 
$\psi$ and $v$ as in the statement of the theorem,
\begin{align*}
\pair{\bc f}{\psi}&=\pair{\ce f}{\dbar \psi}
= \lim_{\epsilon\downarrow 0} \int_\Omega f_\epsilon \dbar \psi
=  \lim_{\epsilon\downarrow 0} \int_\Omega \dbar(f_\epsilon  \psi)
=  \lim_{\epsilon\downarrow 0} \int_\Omega d( f_\epsilon  \psi)
= \lim_{\epsilon\downarrow 0} \int_{\partial\Omega }f_\epsilon  \psi,
\end{align*}
which completes the proof of Theorem~\ref{thm-bcexistence}.

\subsection{Some properties of the boundary current}
The necessity of the Weinstock condition \eqref{eq-weinstock} follows essentially from 
Theorem~\ref{thm-bcexistence}.

 \begin{prop}\label{prop-weinstock} 
 If $\Omega\Subset\Mm$ is a domain with generic corners, and $f\in \mathcal{A}^{-\infty}(\Omega)$. 
 Then, $\bc f$ satisfies the Weinstock orthogonality condition with respect to the domain $\Omega$.
 \end{prop}

\begin{proof}[Proof of Proposition~\ref{prop-weinstock}] From 
\eqref{eq-ce} it follows that if $\phi\in \Dd^{n,n}(\Mm)$ is such that $\phi\equiv 0$ on $\overline{\Omega}$, 
then we have $\pair{\ce f}{\phi}=0$, since for each $\epsilon>0$, we have $\int_\Omega f_\epsilon \phi=0$. Now let $\omega\in \Dd^{n,n-1}(\Mm)$ be such that $\dbar\omega=0$ on $\overline\Omega$. Then,
\begin{align*}
\pair{\bc f}{\omega}&= \pair{-\dbar(\ce f)}{\omega}
=\pair{\ce f}{\dbar \omega}
=0.
\end{align*}
\end{proof}
We also have the following representation of the boundary current along the smooth part of $\partial\Omega$, 
which shows  that on the smooth part, the boundary current is the distributional boundary value written in an 
invariant way: 
\begin{prop}\label{prop-localface} 
Let $U$ be an open subset of $\Mm$ such that $\partial\Omega\cap U$ is a smooth hypersurface in $U$. Then 
there is a  distribution $\alpha\in \Dd'_0(\partial\Omega\cap U)$ such that $\bc f|_U = \iota_*(\alpha)^{0,1}$, 
where $\iota:\partial\Omega\cap U\to U$ is the inclusion map.
\end{prop}

\begin{proof}
Without loss of generality, $U$ meets only one face $\partial\Omega_j$ of $\Omega$. If $\phi\in \Dd^{n,n-1}(\Mm)$ 
has support in $U$, we have, choosing $v$ to be outward from $\Omega$ along $\partial\Omega_j$, and setting 
$f_\epsilon= f(\cdot- \epsilon v)$:
\begin{align*}
\int_{\partial\Omega}f_\epsilon \phi&= \int_{\partial\Omega_j}f_\epsilon \phi
= \pair{\iota_*\left(f_\epsilon|_{\partial\Omega_j}\right)^{0,1}}{\phi},
\end{align*}
where we identify the function $f_\epsilon|_{\partial\Omega_j}$ with the distribution generated by it. Letting $\epsilon\to 0$, we note that for each $\phi$, the left-hand side has a limit, therefore, $\lim_{\epsilon\downarrow 0}f_\epsilon|_{\partial\Omega_j}$ also exists and can be taken to be $\alpha$.
\end{proof}
As a consequence we have the following:
\begin{prop}\label{prop-bcinjective}
For a domain $\Omega$ with generic corners, the map 
$\bc:\mathcal{A}^{-\infty}(\Omega)\to \mathcal{X}^{0,1}_\Omega(\Mm)$ is injective.
\end{prop}

\begin{proof}Let $f\in\mathcal{A}^{-\infty}(\Omega) $ be in the kernel of the linear mapping $\bc$, i.e., 
$\bc f=0$. Since the function $f$ has distributional boundary value 0 in a neighborhood of any smooth point of 
$\partial\Omega$, we see, using the jump representation for distributional boundary values 
\cite[Theorem~6.1, part 5]{kytbook}, that $f$ is,  in fact, smooth up to the boundary near this point. Therefore, 
$f$ continuously assumes the boundary value 0 in an open set on the boundary, which shows that $f=0$.
\end{proof}

\section{Smoothly bounded domains}
\subsection{Proof of Theorem~\ref{thm-smoothcase}}
First, note that if $f\in \mathcal{A}^{-\infty}(\Omega)$, then $\bc f\in \mathcal{X}_{\Omega}^{0,1}(\Mm)$.
Indeed, by Proposition~\ref{prop-weinstock}, the current $\bc f$ satisfies the Weinstock condition,
and the existence of a face distribution $\alpha$ such that $\bc f = \iota_*(\alpha)^{0,1}$ follows from Proposition~\ref{prop-localface}.

The injectivity of the map $\bc: \mathcal{A}^{-\infty}(\Omega)\to \mathcal{X}_{\Omega}^{0,1}(\Mm)$ follows from Proposition~\ref{prop-bcinjective}. Now we show that $\bc:\mathcal{A}^{-\infty}(\Omega)\to \mathcal{X}^{0,1}_\Omega(\Mm)$ is surjective.
Without loss of generality, assume $\Mm$ is connected. Let $\gamma\in \mathcal{X}^{0,1}(\Mm)$
be arbitrary. By definition, for each $\omega\in \Dd^{n,n-1}(\Mm)$ such that $\dbar\omega=0$ on $\overline{\Omega}$, we have
$\pair{\gamma}{\omega}=0$, and  there is  $\alpha\in \Dd'_0(\partial \Omega)$, such that 
$ \gamma= \iota_*(\alpha)^{0,1},$
where $\iota:\partial \Omega\hookrightarrow \Mm$ is the inclusion map. 
We need to show that there is an $f\in \mathcal{A}^{-\infty}(D)$ such that 
$\gamma=\bc f.$

We claim that without loss of generality, we may assume that $\Mm$ is non-compact, and that no component of $\Mm\setminus \Omega$ is compact. To see this,
let $\Nn$ be a noncompact open  submanifold of $\Mm$ defined in the following way. Let $\{D_i\}_{i\in I}$ be the collection 
of relatively compact connected components of $\Mm\setminus \overline{\Omega}$. (The collection  $\{D_i\}_{i\in I}$ may of course be empty.)  For each $i\in I$, fix a point $z_i\in D_i$,
and let $\Nn=\Mm \setminus \{z_i\}_{i\in I}$. Then $\Nn$ is a connected noncompact complex manifold, 
$\Omega\Subset \Nn$,  and the complement $\Nn\setminus\Omega$ does not have any compact components.
Clearly, $\gamma|_{\Nn}\in \mathcal{X}^{0,1}(\Nn)$. The claim follows if we replace $\Mm$ by $\Nn$, and 
$\gamma$ by $\gamma|_{\Nn}$.

By a classical result of Malgrange (see \cite[page 236, comments following Probl\`{e}me 1]{mal57}), the Dolbeault cohomology 
group $H^{n,n}(\Mm)$ vanishes, since $\Mm$ is noncompact and connected.  Now, the transpose of the surjective linear continuous map of Fr\'{e}chet spaces
$\dbar: \Ee^{n,n-1}(\Mm)\to \Ee^{n,n}(\Mm)$  (where $\Ee^{p,q}(\Mm)$ is the space of smooth $(p,q)$-forms on $\Mm$) can be identified with  the map $\dbar: \Ee'_{0,0}(\Mm)\to \Ee'_{0,1}(\Mm)$, where 
$\Ee'_{p,q}(\Mm)$ is the space of compactly supported $(p,q)$-currents on $\Mm$ (see \cite[Proposition 5]{serre}).  It now follows from a well-known result of functional analysis (see \cite[Theorem~37.2]{trevesbook}) that the range of $\dbar: \Ee'_{0,0}(\Mm)\to \Ee'_{0,1}(\Mm)$ is closed in $\Ee'_{0,1}(\Mm)$ with respect to its weak topology,  and therefore with respect to the usual strong topology.  Therefore, the range of $\dbar: \Ee'_{0,0}(\Mm)\to \Ee'_{0,1}(\Mm)$ can be identified with the subspace of $\Ee'_{0,1}(\Mm)$ orthogonal to the kernel of $\dbar: \Ee^{n,n-1}(\Mm)\to \Ee^{n,n}(\Mm)$, i.e., the range of $\dbar: \Ee'_{0,0}(\Mm)\to \Ee'_{0,1}(\Mm)$ consists precisely of those 
$\theta\in  \Ee'_{0,1}(\Mm)$ which have the property that for each $\omega\in \Ee^{n,n-1}(\Mm)$ such that $\dbar\omega=0$ we have $\pair{\theta}{\omega}=0$. It now follows from the Weinstock orthogonality condition \eqref{eq-weinstock} that $\gamma$ lies in the range of $\dbar: \Ee'_{0,0}(\Mm)\to \Ee'_{0,1}(\Mm)$, so that there is a
compactly distribution $u$ on $\Mm$ such that $\dbar u =-\gamma$.

From the  structure of $\gamma$ given by \eqref{eq-facedistn}, it follows that the support of $\gamma$ is contained in the boundary $\partial \Omega$. Therefore, $u$ is holomorphic on $\Mm\setminus\partial\Omega$, in particular,
it is holomorphic on $\Mm\setminus \overline{\Omega}$. However, since $u$ has compact support, and no component of $\Mm\setminus\overline{\Omega}$ is relatively compact by assumption, it follows that $u$ vanishes on an open subset of each component of  $\Mm\setminus\overline{\Omega}$.  Therefore, by analytic continuation, $u$ vanishes on each component of $\Mm\setminus \overline{\Omega}$, and so the support 
of $u$ is contained in $\overline{\Omega}$.

We set $f=u|_{\Omega}$. Since $u$ is holomorphic on $\Mm\setminus \partial\Omega$, it follows that $f\in \mathcal{O}(\Omega)$. Since the holomorphic function $f$ on $\Omega$ can be extended to the distribution $u$ on 
$\Mm$,  it follows by Proposition~\ref{prop-extension} (with $\Xx=\cx$) that $f\in \mathcal{A}^{-\infty}(\Omega)$.
The surjectivity of $\bc$ will be established if we show that $\gamma=\bc f$.
This is clearly a local question, 
so we pick a point $p\in \partial\Omega$, and a system of holomorphic coordinates around $p$. Let $B$ be a ball in these coordinates centered at $p$, so small that $\partial \Omega$ divides $B$ into two pieces $B^-=B\cap \Omega$ and $B^+=B\setminus \overline{\Omega}$. Now since $H^{0,1}(B)=0$, we can solve the $\dbar$-problem $\dbar h = -\gamma|_B$ on $B$, 
and the solution $h$ may be represented by a Bochner-Martinelli type integral (see \cite[Chapter 6]{kytbook}).
Then $h$  has the property that for each $\phi\in \mathcal{D}^{n,n-1}(B)$, we have (see \cite[Theorem~6.1, part 5]{kytbook}):
\[ -\gamma(\phi)= \lim_{\epsilon\downarrow 0}\int_{B\cap\partial \Omega}\left( h(\zeta+\epsilon\nu(\zeta))-h(\zeta-\epsilon\nu(\zeta))\right)\phi(\zeta),\]
where $\nu$ is the outward unit normal vector field on $\partial\Omega\cap B$.
Further, since $h-u$ is holomorphic on $B$, after subtracting a holomorphic function on $B$ from $h$, 
we can assume that $h=u$.  Therefore,  $h(\zeta+\epsilon\nu(\zeta))=0$ and $h(\zeta-\epsilon\nu(\zeta))=f(\zeta-\epsilon\nu(\zeta))$, and we have for each $\phi\in \mathcal{D}^{n,n-1}(B)$,
\[ \gamma(\phi)= \lim_{\epsilon\downarrow 0}\int_{\partial \Omega}f(\zeta-\epsilon\nu(\zeta))\phi(\zeta),\]
so that we have  $ \gamma= \bc f$,
which shows that $\bc$ is surjective.

Therefore $\bc: \mathcal{A}^{-\infty}(\Omega)\to \mathcal{X}^{0,1}_\Omega(\Mm)$ is a continuous bijection  of 
topological vector spaces. We know from Proposition~\ref{prop-ainftydfs} that $\mathcal{A}^{-\infty}(\Omega)$ is a DFS space, and $\mathcal{X}^{0,1}_\Omega(\Mm)$ is also a DFS space, since it is a closed subspace of the DFS space $\Dd'_0(\partial\Omega)$. The fact that
$\bc$ is an isomorphism of topological vector spaces now follows from Prop~\ref{res-dfsinverse}.

\begin{proof}[Proof of Corollary~\ref{cor-bochnerhartogs}] 
In view of Theorem~\ref{thm-smoothcase} above it suffices to show that $\gamma\in \mathcal{X}^{0,1}_\Omega(\cx^n)$. For this we will show that  $\gamma$ satisfies the Weinstock criterion 
with respect to $\Omega$ if and only if $\partial\overline\gamma =0$.  

If $\gamma$ satisfies\eqref{eq-weinstock}, then take $\omega=\dbar\phi$ for $\phi\in \Dd^{n,n-2}(\Mm)$. Condition \eqref{eq-weinstock} now implies that $\dbar\gamma=0$. Suppose now that $\partial\overline\gamma =0$. Let $\phi\in \Dd^{n,n-1}(\cx^n)$ be such that $\dbar\phi=0$ on $\Omega$. By
 an approximation result for $(n,n-1)$-forms (see \cite[Theorem~1]{weinstock2}), there is a sequence of forms $\psi_\nu$
 in $\Dd^{n,n-2}(\cx^n)$ such that $\dbar \psi_\nu\to \phi$ on $\overline{\Omega}$ in the Fr\'{e}chet space $\mathcal{C}^\infty_{n,n-1}(\overline{\Omega})$. Therefore, noting that $\gamma$ is supported in $\overline{\Omega}$, we have
 \begin{align*}
 \pair{\gamma}{\phi}&= \lim_{\nu\to \infty} \pair{\gamma}{\dbar \psi_\nu}
 = -\lim_{\nu\to \infty} \pair{\dbar\gamma}{\psi_\nu}=0.
 \end{align*} \end{proof}

\section{Boundary currents on product domains}

\subsection{Notation and terminology}
Throughout this section and the next $\Mm_j$, $\Mm$, $\Omega$ and $\Omega_j$ will have the same meanings as in \eqref{eq-mproduct}, \eqref{eq-omegaproduct} and \eqref{eq-omegaj} respectively. We introduce some notation to describe the geometry of $\Mm$ and $\Omega$. We set $n_j=\dim_\cx \Mm_j$.
For $j=1,\dots, N$,  let $\pi_j:\Mm \to \Mm_j $
be the natural projection onto the $j$-th factor.
 If for each $j$, $\phi_j$ is a form on $\Mm_j$, then we define
\begin{equation}\label{eq-formtensor}
\phi_1\tensor \dots\tensor\phi_N = \pi_1^* \phi_1\wedge \pi_2^* \phi_2 \wedge \dots \wedge \pi_N^*\phi_N.
\end{equation}
For $j=1,\dots, N$, let $(p_j,q_j)$ be ordered pairs of integers with $0\leq p_j, q_j\leq n_j$, and let $P=\sum_{j=1}^N p_j, Q=\sum_{j=1}^N q_j$. We denote the $\cx$-linear span of the forms
$\Set{\phi_1\tensor \dots\tensor\phi_N| \phi_j\in \Dd^{p_j,q_j}(\Mm_j)}$
by 
\begin{equation}\label{eq-algtens}
\Dd^{p_1,q_1}(\Mm_1)\tensor\dots\tensor\Dd^{p_N,q_N}(\Mm_N),
\end{equation}
which is the  {\em algebraic tensor product} of the spaces $\Dd^{p_j,q_j}(\Mm_j)$. The closure of the 
space \eqref{eq-algtens} in $\Dd^{P,Q}(\Mm)$ is the {\it topological} 
tensor product of the spaces  $\Dd^{p_j,q_j}(\Mm_j)$ and is denoted by
\[ \Dd^{p_1,q_1}(\Mm_1)\csor\dots\csor\Dd^{p_N,q_N}(\Mm_N).\]
We can define tensor products of spaces of currents in the same way.
 
 We will denote by $\widehat{\Mm}_j$ the product of all the $\Mm_k$ except $\Mm_j$, i.e., 
$\widehat{\Mm}_j=\Mm_1\times\dots\times\Mm_{j-1}\times\Mm_{j+1}\times \dots\times\Mm_N,
$
and then we  will write $ \Mm =\Mm_j \times \widehat{\Mm}_j$. Throughout the paper
we will assume that all products  are reordered in the standard way, i.e., the factor $\Mm_j$ 
is to be inserted into the  slot between $\Mm_{j-1}$ and $\Mm_{j+1}$.  Similarly, we can write 
\eqref{eq-omegaj} as $\Omega_j = D_j\times \widehat{\Mm}_j$, and we have
$\partial\Omega_j=\partial D_j\times\widehat{\Mm}_j$, keeping in mind the reordering. 
We will similarly keep the same notation $\csor$ for a ``reordered'' tensor product, for example
\begin{align*} \Dd'_0(\partial\Omega_j)&=\Dd'_0(\partial D_j\times \widehat{\Mm}_j)\\
&\cong \Dd'_0(\partial D_j)\csor\Dd'_0(\widehat{\Mm}_j) ,
\end{align*}
where in the last line we have used the Schwartz kernel theorem (see \cite[Theorem~51.7]{trevesbook}), and the tensor product is reordered.
In this notation, one can  write down the direct sum decomposition
\begin{equation}\label{eq-01dec}
\mathcal{D}'_{0,1}(\Mm) = \bigoplus_{j=1}^N \mathcal{D}'_{0,1}({\Mm}_j)\csor \mathcal{D}'_{0,0}(\widehat{\Mm}_j),
\end{equation}
which is easily established using the degree considerations. Given a current $\gamma\in \Dd'_{0,1}(\Mm)$ 
we can therefore write uniquely
\begin{equation}\label{eq-01dec2}
\gamma= \sum_{k=1}^N \gamma_k, \ \  \text{ with  }
\gamma_k\in \Dd'_{0,1}(\Mm_k)\csor \mathcal{D}'_{0,0} (\widehat{\Mm}_k) .
\end{equation}
We will refer to \eqref{eq-01dec2} as the {\em standard decomposition} of a (0,1)-form on the product manifold $\Mm$.

 We denote by $\widehat{D}_j$ the domain in $\widehat{\Mm}_j$ which is the product of all the $D_k$ except $D_j$:
\[ \widehat{D}_j=D_1\times\dots\times D_{j-1}\times D_{j+1}\times \dots\times D_N,
 \]
 and we call $\partial D_j\times \widehat{D}_j$ the $j$-th {\em open face} of $\Omega$. Its closure is denoted 
 by $F_j$, and will be called the {\em $j$-th face} of $\Omega$. Then $F_j$ has the representations
 \begin{align*}F_k&=\partial D_k \times \overline{\widehat{D}_k}\\ 
&= (\partial D_k\times\widehat{\Mm}_k)\cap (\overline{D_k}\times \overline{\widehat{D}_k})\\
&= \partial\Omega_k\cap \overline{\Omega} .
\end{align*}
  Note that the open face $\partial D_j\times \widehat{D}_j$ and the face $F_j$   are subsets of the manifold $\partial\Omega_j$, which we will call the $j$-th {\em extended face} of $\Omega$.
  
We let $\jmath^k:\partial D_k\to \Mm_k$ denote the inclusion map. If $\iota^k:\partial\Omega_k\to \Mm$ denotes 
the inclusion map of the extended face $\partial\Omega_k$ in the product manifold $\Mm$, we can write:
\begin{equation}\label{eq-iotak}
\iota^k= \jmath^k\times \id ,
\end{equation}
where $\id:\widehat{\Mm}_k \to\widehat{\Mm}_k$ denotes  the identity map, and  $\times$ is the reordered direct product of {\em maps}  (with respect to the reordering $\Mm= \Mm_k \times \widehat{\Mm}_k$).
Then we can define  pushforward maps $\jmath^k_*:\Dd'_0(\partial D_k)\to \Dd'_1(\Mm_k)$ , and  $\iota^k_*:\Dd'_0(\partial\Omega_k)\to \Dd'_1(\Mm)$ as in \eqref{eq-pushforward}. Thanks to \eqref{eq-iotak}, we see that these two are related by
\begin{equation}\label{eq-iotatensor}
\iota^k_*= \jmath^k_*\csor\id .
\end{equation}
Here $\id$ denotes the identity on $\Dd'_0(\widehat{\Mm}_k)$, and $\csor$ has an obvious meaning as a reordered tensor product of continuous maps of topological vector spaces. 

\subsection{Some computations with the $\dbar$-operator on a product manifold}
From the Cauchy-Riemann operators $\dbar:\Dd'_*(\Mm_k)\to \Dd'_*(\Mm_k)$ and $\dbar:\Dd'_*(\widehat{\Mm}_k)\to  \Dd'_*(\widehat{\Mm}_k)$ we can construct two ``partial Cauchy-Riemann operators''
$\dbar_{\Mm_k}: \Dd'_*(\Mm)\to \Dd'_*(\Mm)$ and $\dbar_{\widehat{\Mm}_k}:\Dd'_*(\Mm)\to \Dd'_*(\Mm)$ using the reordered product representation $\Mm=\Mm_k\times \widehat{\Mm}_k$ and setting
\begin{equation}\label{eq-dbarmk}
\dbar_{\Mm_k}= \dbar\csor \id,
\end{equation}
and
\begin{equation}\label{eq-dbaralongfibers}
 \dbar_{\widehat{\Mm}_k} = {\rm \id_k} \csor \dbar.
 \end{equation}
 In \eqref{eq-dbarmk},  $\dbar$ denotes the $\dbar$-operator on $\Mm_k$, whereas $\id$ is the identity on $\widehat{\Mm}_k$. On the other hand, in \eqref{eq-dbaralongfibers}, $\dbar$ is the $\dbar$-operator on $\widehat{\Mm}_k$ and $\id_k$ denotes the identity on $\Mm_k$.
Intuitively, $\dbar_{\Mm_k}$ takes the derivative along the factor $\Mm_k$ only and $\dbar_{\widehat{\Mm}_k}$ 
takes the derivative along the factor $\widehat{\Mm}_k$.

\begin{prop}\label{prop-proddbarclosed}Let $\gamma\in \Dd'_{0,1}(\Mm)$ be such that $\dbar\gamma=0$ and let $\gamma_k$ be as in \eqref{eq-01dec2}.	  Then,
\begin{enumerate}
\item for $k=1,\dots, N$, we have
\begin{equation}\label{eq-TCR}
\dbar_{\Mm_k}\gamma_k=0,
\end{equation}
\item  for $j\not =k$, with $j,k=1,\dots, N$, we have
\begin{equation} \label{eq-edge}
\dbar_{\Mm_j}\gamma_k - \dbar_{\Mm_k}\gamma_j=0.
\end{equation}
\end{enumerate}
\end{prop}
To prove this, for $j=1,\dots, N$, define an operator $\sigma_j$ on the space of forms on $\Mm$ in the following way. We define $\sigma_1$ to be the identity operator and if $j\geq 2$, for forms
$\phi_k$ on the $\Mm_k$ of a fixed bidegree, we set
\[ \sigma_j (\phi_1\tensor\dots \tensor \phi_N)= \left(-1\right)^{ \sum_{k=1}^{j-1}\deg (\phi_k)} \phi_1\tensor \dots\tensor \phi_N,\]
and then extend by linearity and continuity to $\Dd'_*(\Mm)$. 
\begin{lem}On $\Mm$, we have
\begin{equation}\label{eq-dbarprod}
\dbar = \sum_{j=1}^N \sigma_j \dbar_{\Mm_j} .
\end{equation}
\end{lem}
\begin{proof}By linearity and density, it is sufficient to show this for tensor products of forms of a fixed bidegree. The statement is obvious for $N=1$, so we assume it for $N-1$, and prove it for $N$, which 
gives the proof by induction.
Let $\phi= \phi_1\tensor\dots\tensor \phi_N$, where $\phi_j$ is of fixed bidegree on $\Mm_j$. Also let
$\phi'= \phi_1\tensor\dots\tensor \phi_{N-1}$. Using the formula for the exterior derivative of a wedge product,
we have
\begin{align*}
\dbar\phi &= \dbar\phi'\tensor \phi_N + (-1)^{\deg \phi'} \phi'\tensor \dbar\phi_N\\
&= \left( \sum_{j=1}^{N-1} \sigma_j \dbar_{\Mm_j}\phi'\right)\tensor \phi_N + \sigma_N (\phi'\tensor \dbar\phi_N)\\
&= \left( \sum_{j=1}^{N} \sigma_j \dbar_{\Mm_j}\right)\phi.
\end{align*}
\end{proof}
\begin{proof}[Proof of Proposition~\ref{prop-proddbarclosed}]
We write $\gamma= \sum_{k=1}^N \gamma_k$. Then we claim that 
\[ \sigma_j \dbar_{\Mm_j}\gamma_k = \begin{cases}\phantom{-}\dbar_{\Mm_j}\gamma_j   & \text{if $j\leq k$}\\ 
-\dbar_{\Mm_j}\gamma_j  & \text{if $j>k$.}\end{cases} \]
Indeed, it is sufficient to consider the case when $\gamma_k=\alpha_1\tensor\dots\tensor \alpha_N$, where 
each $\alpha_\ell$ is of degree 0 except $\alpha_k$ which is of degree $(0,1)$. Taking the $\dbar$ with respect to $\Mm_j$, 
we see that the first $j-1$ factors in the representation of $\dbar\gamma_k$ as a tensor product are all of degree 0 if 
$j\leq k$, but contains the single factor $\alpha_k$ of degree 1 if $j>k$. The claim follows from the 
definition of $\sigma_j$. Therefore,
\begin{align}
\dbar\gamma &= \left(\sum_{j=1}^N\sigma_j \dbar_{\Mm_j}  \right) \left(\sum_{k=1}^N \gamma_k  \right)\nonumber\\
&= \sum_{j=1}^N\sum_{k=1}^N\sigma_j\dbar_{\Mm_j}\gamma_k\nonumber\\
&= \sum_{k=1}^N \dbar_{\Mm_k}\gamma_k + \sum_{j<k} \left(\dbar_{\Mm_j}\gamma_k - \dbar_{\Mm_k}\gamma_j\right).\label{eq-dbargammadec}
\end{align}
If $j\not=k$, let $\widehat{\Mm}_{j,k}$ denote the product of all the factors of $\Mm$ except $\Mm_j$ and $\Mm_k$. We then have the direct sum decomposition for
$(0,2)$-currents:
\begin{equation}\label{eq-02dec}
\Dd'_{0,2}(\Mm)= \bigoplus_{k=1}^N \Dd'_{0,2}(\Mm_k)\csor \Dd'_{0,0}(\widehat{\Mm}_k) \oplus
\left(\bigoplus_{j<k} \Dd'_{0,1}(\Mm_j)\csor \Dd'_{0,1}(\Mm_k) \csor \Dd'_{0,0}(\widehat{\Mm}_{j,k}) \right),
\end{equation}
where the tensor products are reordered so that each direct summand is a subspace of $\Dd'_{0,2}(\Mm)$.
Note that the terms in \eqref{eq-dbargammadec} correspond to the direct summands in the decomposition \eqref{eq-02dec}, which must each vanish since $\dbar\gamma=0$.
\end{proof}
We can represent the $\dbar_{\widehat{\Mm}_k}$ operator of \eqref{eq-dbaralongfibers}	 in terms of the $\partial_{{\Mm}_j}$. Indeed,  we can 
show that
\[ \dbar_{\widehat{\Mm}_k} = \sum_{j\not=k} \sigma_{jk}\overline\partial_{{\Mm}_j},\]
where $\sigma_{jk}\in \{\pm 1\}$. (The precise sign of $\sigma_{jk}$, while not difficult to find, is irrelevant for the intended application.)
 \begin{prop} Let $\gamma$ and $\gamma_k$ be as in Proposition	~\ref{prop-proddbarclosed}. Then 
\begin{equation}\label{eq-support}
\left.\left(\dbar_{\widehat{\Mm}_k}\gamma_k\right)\right|_{\partial D_k\times\widehat{D}_k}=0.
\end{equation}
\end{prop}
\begin{proof}  By conclusion (2) 
of  Proposition~\ref{prop-proddbarclosed} we have for $j\not=k$ that $\dbar_{\Mm_j}\gamma_k =\dbar_{\Mm_k}\gamma_j.$ 
Consequently, $\dbar_{\Mm_j}\gamma_k=0$ outside $\supp(\dbar_{\Mm_k}\gamma_j)\subset\supp(\gamma_j)$.  Therefore, outside the
set $\bigcup_{j\not=k} \supp(\gamma_j)$ we have
\[ \dbar_{\widehat{\Mm}_k}\gamma_k = \sum_{j\not=k} \sigma_{jk}\partial_{{\Mm}_j}\gamma_k=0.\]
The statement follows on noting that $\partial D_k\times\widehat{D}_k$ is disjoint from $\supp(\gamma_j)\subset \partial D_j\times \widehat{\Mm_j}$ for each $j\not =k$.
\end{proof}

 \subsection{Existence of face distributions} 
 We begin with the following Lemma:
 
 \begin{lem}\label{lem-facestruct} 
 Let $\Nn$ be a smooth manifold of dimension $n$,  let $S$ be a smooth hypersurface in $\Nn$, and let $\iota:S\to \Nn$ be the inclusion map and let $\gamma\in \Dd'_1(\Nn)$.
 Then there is a distribution $\alpha\in \Dd'_0(S)$ such that $\iota_*(\alpha)=\gamma$ if and only if for each $\phi\in \Dd^{n-1}(\Nn)$ such that $\iota^*\phi=0$, we have $\pair{\gamma}{\phi}=0$.
 Consequently, the subspace of currents in $\Dd'_1(\Nn)$ of the form $\iota_*(\alpha)$ with $\alpha\in \Dd'_0(S)$ is closed. 
 \end{lem}
 
 \begin{proof} 
 If $\gamma= \iota_*\alpha$, then for any $\phi\in \Dd^{n-1}(\Nn)$ with $\iota^*\phi=0$, we have $\pair{\gamma}{\phi}= \pair{\alpha}{\iota^*\phi}=0$. On the other hand, there is a continuous linear  extension operator $E:\Dd^{n-1}(S)\to \Dd^{n-1}(\Nn)$, such that for each $\psi\in \Dd^{n-1}(S)$, we have $\iota^*(E\psi)=\psi$. The existence of $E$ is obvious locally using coordinates, and follows globally using a partition of unity argument. We define $\alpha\in \Dd'_0(S)$ by $\pair{\alpha}{\psi}= \pair{\gamma}{E\psi}$.  Note that $\alpha$ is independent of the particular continuous extension operator 
 $E$. Indeed, if $\widetilde{\psi}\in\Dd^{n-1}(\Nn) $ is any other extension of $\psi\in \Dd^{n-1}(S)$ 
 (i.e., $\iota^*\widetilde{\psi}=\psi$), then $\pair{\gamma}{E\psi}-\pair{\gamma}{\widetilde{\psi}}= \pair{\gamma}{E\psi-\widetilde{\psi}}=0$, since $\iota^*(E\psi-\widetilde{\psi})=0$.  For the $\alpha$ so defined,  and any $\phi\in \Dd^{n-1}(\Nn)$, we have $\pair{\iota_*\alpha}{\phi}=\pair{\alpha}{\iota^*\phi}=\pair{\gamma}{\phi}$, since $\phi$ is an extension of $\iota^*\phi$.  
 \end{proof}
 \begin{prop}\label{prop-bcfacewise}
 Let $f\in \mathcal{A}^{-\infty}(\Omega)$. Then there are distributions $\alpha_k\in \Dd'_0(\partial\Omega_k)$ supported on $F_k=\partial\Omega_k\cap \overline{\Omega}$ such that
 \[ \bc f = \sum_{k=1}^N \iota^k_*(\alpha_k)^{0,1},\]
where the notation is as in \eqref{eq-iotak} and \eqref{eq-iotatensor}. Further, the $k$-th summand on 
the right-hand side is 
precisely the $k$-th component of the standard decomposition \eqref{eq-01dec2} and lies in $ \Dd'_{0,1}(\Mm_k)\csor \Dd'_0(\widehat{\Mm}_k).$ 
  \end{prop}
 \begin{proof}
 Thanks to the decomposition \eqref{eq-01dec} of (0,1)-forms on a product manifold, we can write $\bc f = \sum_{k=1}^N \gamma_k$, where $\gamma_k \in \Dd'_{0,1}(\Mm_k)\csor \Dd'_0(\widehat{\Mm}_k).$ We identify the summands $\gamma_k$.  By definition of $\bc f$, 
each $\gamma_k$ is supported on $\partial \Omega$.  We cover $\Omega$ by open sets $U$ of the type considered in  Theorem~\ref{thm-bcexistence}, i.e., $U$ is a coordinate neighborhood of $\Mm$, and there is a vector $v\in \cx^n$ such that in the coordinates on $U$, the vector $v$ points outward from $\Omega$ along each $\partial\Omega_j$   meeting $U$. Fix one such $U$, and let $\phi \in \Dd^{n,n-1}(U)$ be a form of degree $(n,n-1)$ with compact support in $U$. For $f_\epsilon$ as in Theorem~\ref{thm-bcexistence}, we have
 \begin{align}
 \pair{\bc f|_U}{\phi}&= \lim_{\epsilon\downarrow 0}\int_{\partial\Omega} f_\epsilon\phi\nonumber\\
 &=   \sum_{k=1}^N \lim_{\epsilon\downarrow 0}\int_{\partial\Omega_k} f_\epsilon|_{\overline{\Omega}\cap U}\cdot\phi\nonumber\\
&= \sum_{k=1}^N \lim_{\epsilon\downarrow 0} \pair{f_\epsilon[\overline{\Omega}\cap U\cap\partial\Omega_k]}{(\iota^k)^*\phi}_{\partial\Omega_k}\nonumber\\
&= \sum_{k=1}^N \lim_{\epsilon\downarrow 0} \pair{\iota^k_*\left(f_\epsilon[F_k\cap U]\right)^{0,1}}{\phi}.\label{eq-bcfcomp}
\end{align}
here  $F_k= \overline{\Omega}\cap\partial\Omega_k$  is the $k$-th face of $\Omega$, and $[F_k\cap U]\in \Dd'_0(\partial\Omega_k)$ denotes the 0-current of integration on the set $F_k\cap U$.  By Lemma~\ref{lem-facestruct} above, the subspace of  $\Dd'_0(\Mm)$ consisting 
 of currents of the type $\iota^k_*(\beta)$, (where $\beta\in \Dd'_0(\partial\Omega_k)$) is closed, so it easily follows that 
 there is an $\alpha_k^U\in \Dd'_0(\partial\Omega_k)$ such that $\gamma_k|_U= \iota^k_*(\alpha_k^U)^{0,1}$. The existence of $\alpha_k \in \Dd'_0(\partial \Omega_k)$ such that $\gamma_k= \iota^k_*(\alpha_k)^{0,1}$ now follows by a partition of unity argument.
From the representation \eqref{eq-bcfcomp}, it follows that $\alpha_k$ has support in the subset $F_k$ of $\partial \Omega_k$.
 \end{proof}
 
 \subsection{Currents with facewise structure} 
In view of Proposition~\ref{prop-bcfacewise} above, we make the following 
 definition: let $\Mm$ be a product manifold as in \eqref{eq-mproduct}, and let $\Omega\Subset\Mm$ be a product domain
 as in \eqref{eq-omegaproduct}. We say that a current $\gamma\in \Dd'_{0,1}(\Mm)$ has {\it facewise structure} with  respect to $\Omega$, if there are {\it face distributions} $\alpha_j\in \Dd'_0(\partial\Omega_j)$, for 
 $j=1,\dots, N$, with support in $F_j$, such that 
 \begin{equation}\label{eq-gammafacewise}
\gamma=\sum_{j=1}^N \iota^j_*(\alpha_j)^{0,1},
\end{equation}
 where $\iota^j:\partial\Omega_j\to \Mm$ is the inclusion map. Then, Proposition~\ref{prop-bcfacewise} states 
 that the boundary current of a holomorphic function of polynomial growth on a product domain has facewise structure. 
 
\begin{prop}\label{prop-alpha1alpha2}
Let $\Mm$ and $\Omega$ be as in \eqref{eq-mproduct}, \eqref{eq-omegaproduct}. Let $\gamma\in \Dd'_{0,1}(\Mm)$
satisfy the Weinstock condition and have facewise structure \eqref{eq-gammafacewise}, both with respect to $\Omega$.
Then for each $k\in \{1,\dots, N\}$,  we have 
\begin{equation}\label{eq-alphajstruct}
\alpha_k|_{\partial D_k\times \widehat{D}_k} \in \Dd'_0(\partial D_k)\csor \mathcal{A}^{-\infty}(\widehat{D}_k).
\end{equation}
\end{prop}

Intuitively, this says that the distributions $\alpha_k$, which may be thought of ``restrictions'' of $\gamma$ to the faces, are each holomorphic and of polynomial growth along the complex factor, which is clearly the case for continuous boundary values on a product domain. Note also that in view of the fact that $\gamma_k = \iota^k_*(\alpha_k)$, the relation \eqref{eq-alphajstruct} is equivalent to 
\begin{equation}\label{eq-gammakstruct}
\gamma_k|_{\Mm_k\times \widehat{D}_k} \in \Dd'_{0,1}(\Mm_k)\csor \mathcal{A}^{-\infty}(\widehat{D}_k).
\end{equation}
\begin{proof}From \eqref{eq-iotatensor},	 it follows that in the decomposition of   $\gamma\in \Dd'_{0,1}(\Mm)$ into direct summands given by \eqref{eq-01dec}, the component $\gamma_k\in \Dd'_{0,1}(\Mm_k)\csor \Dd'_0(\widehat{\Mm}_k)$  is given by 
\begin{equation}\label{eq-gammakfacewise}
 \gamma_k = (\jmath^k_*\csor \id)(\alpha_k)^{0,1},
\end{equation}
where $\alpha_k\in \Dd'_0(\partial\Omega_k)$ is the $k$-th  face distribution associated with the current $\gamma$.
Combining \eqref{eq-gammakfacewise} and \eqref{eq-dbaralongfibers}, we see that 
\begin{align*}\dbar_{\widehat{\Mm}_k}\gamma_k &=(\id\csor\dbar)(\jmath^k_*\csor \id)(\alpha_k)^{0,1}\\
&=((\jmath^k_*\csor \dbar)(\alpha_k))^{0,1}\\
&=(\jmath^k_*\csor \dbar)(\alpha_k).
\end{align*}
Therefore, using \eqref{eq-support} we see that 
on the manifold $\partial D_k \times \widehat{\Mm}_k$, the current $(\id \csor \dbar)(\alpha_k)$ vanishes on the open set
$\partial D_k \times \widehat{D}_k$, i.e., $\alpha_k$ is holomorphic in the direction of $\widehat{\Mm}_k$ on the open set $\partial D_k \times \widehat{D}_k$ It follows that 
\[ \alpha_k|_{\partial D_k\times \widehat{D}_k} \in \Dd'_0(\partial D_k)\csor \mathcal{O}(\widehat{D}_k).\]

Now we note that the face distribution  $\alpha_k\in \Dd'_0(\partial \Omega_k)$ (where we recall that 
$\partial \Omega_k= \partial D_k\times \widehat{\Mm}_k$) is supported in $\partial D_k\times \overline{\widehat{D}_k}$. 
Consequently,  $\alpha_k$ is a distribution of finite order on $\partial\Omega_k$. It follows that there is an integer  $K$, such that 
\[ \alpha_k|_{\partial D_k\times \widehat{D}_k}\in \mathcal{C}^{-K}(\partial D_k)\csor \mathcal{O}(\widehat{D}_k)\cong \mathcal{O}(\widehat{D}_k,\mathcal{C}^{-K}(\partial D_k)),\]
where $\mathcal{C}^{-K}(\partial D_k)$ is the Banach space of distributions of order $K$ on $\partial D_k$ (it is a Banach space since $\partial D_k$ is compact), and the  isomorphism of  the space  $\mathcal{O}(\widehat{D}_k,\mathcal{C}^{-K}(\partial D_k))$
 with the topological tensor product $\mathcal{C}^{-K}(\partial D_k)\csor \mathcal{O}(\widehat{D}_k)$ (which makes sense 
since $\mathcal{O}(\widehat{D}_k)$ is nuclear) follows as in \cite[Theorem~44.1]{trevesbook}. Using this isomorphism, 
interpreting  $\alpha_k |_{\partial D_k\times \widehat{D}_k}$
 as a Banach-valued holomorphic
function on $\widehat{D}_k$, we see that it can be extended to a Banach-valued distribution on $\widehat{\Mm}_k$.
Therefore, by Proposition~\ref{prop-extension}, we have $\alpha_k\in \mathcal{A}^{-\infty}(\widehat{D}_k,\mathcal{C}^{-K}(\partial D_k))$.
 Since each distribution on the compact manifold $\partial D_k$ is of finite order, we see that 
\[ \alpha_k|_{\partial D_k\times \widehat{D}_k}\in  \Dd'_0(\partial D_k)\csor \mathcal{A}^{-\infty}(\widehat{D}_k),\]
which proves \eqref{eq-alphajstruct}. \end{proof}

\subsection{The space $\mathcal{Y}^{0,1}_\Omega(\Mm)$}\label{sec-y01}We will now state precisely the third condition
in the definition of the space $\mathcal{Y}^{0,1}_\Omega(\Mm)$. Let $\Mm$ and $\Omega$ be as above 
a product manifold and a product domain as in \eqref{eq-mproduct} and \eqref{eq-omegaproduct}. Suppose that a current $\gamma\in \Dd'_{0,1}(\Mm)$ satisfies the Weinstock condition and has facewise structure \eqref{eq-gammafacewise}, both with respect to $\Omega$. Then by Proposition~\ref{prop-alpha1alpha2} the relation \eqref{eq-alphajstruct} holds.
The third condition in the definition of $\mathcal{Y}^{0,1}_\Omega(\Mm)$ is the following: for $k=1,\dots, N$, we have
\begin{equation}
(\id_k\csor \ce_{\widehat{k}}) \left( \alpha_k|_{\partial D_k\times \widehat{D}_k}\right)
=\alpha_k, \label{eq-canonicality}
\end{equation}
where $\id_k$  is the identity map on $\Dd'_0(\partial D_k)$ and $\ce_{\widehat{k}}:\mathcal{A}^{-\infty}(\widehat{D}_k)\to \Dd'_0(\widehat{\Mm}_k)$ is the canonical extension operator. We refer to \eqref{eq-canonicality} as the {\it canonicality 
condition on face distributions.} We note that the condition \eqref{eq-canonicality} may be directly expressed in terms of the 
current $\gamma$ as:
\begin{equation}
(\id_k\csor \ce_{\widehat{k}}) \left( \gamma_k|_{\Mm_k\times \widehat{D}_k}\right)=\gamma_k,\label{eq-canonicality2}
\end{equation}
where $\gamma_k$ is the $k$-th component of $\gamma$ in the standard decomposition \eqref{eq-01dec2}, $\id_k$  now denotes the identity map on $\Dd'_{0,1}(\Mm_k)$ and $\ce_{\widehat{k}}$ is as in \eqref{eq-canonicality}. Also
note that \eqref{eq-canonicality2} makes sense thanks to \eqref{eq-gammakstruct}. To prove \eqref{eq-canonicality2}, we have
\begin{align*}
(\id_k\csor \ce_{\widehat{k}})\left( \gamma_k|_{\Mm_k\times \widehat{D}_k}\right)&=
(\id_k\csor \ce_{\widehat{k}})\left((\jmath^k\csor \id_{\widehat{k}})(\alpha_k|_{\partial D_k\times \widehat{D}_k}) \right)^{0,1}\\
&=\left( (\jmath^k_* \csor \ce_{\widehat{k}})(\id_k \csor \ce_{\widehat{k}})\left(\alpha_k|_{\partial D_k\times \widehat{D}_k}\right)\right)^{0,1}\\
&=\left(\iota_*^k(\alpha_k)\right)^{0,1} &\text{ using \eqref{eq-canonicality}}\\
&=\gamma_k,
\end{align*}
where in the second line, $\id_k$ denotes the identity operator on $\partial D_k$.  The converse implication, i.e., that \eqref{eq-canonicality2} implies \eqref{eq-canonicality}, can be proved by an analogous computation.

Therefore,  $\mathcal{Y}^{0,1}_\Omega(\Mm)$ consists of those $(0,1)$-currents 
in $\Dd'_{0,1}(\Mm)$ which satisfy the Weinstock criterion, have facewise structure 
(both with respect to $\Omega$), and whose face distributions are canonical in the sense 
of \eqref{eq-canonicality} or, equivalently, \eqref{eq-canonicality2}. Since all three 
conditions are closed, it follows that $\mathcal{Y}^{0,1}_\Omega(\Mm)$ is a closed 
subspace of $\Dd'_{0,1}(\Mm)$. The next proposition shows that in Theorem~\ref{thm-product} 
we have identified the correct target space.

\begin{prop}\label{prop-y01target} Let $\Omega\Subset\Mm$ be a product domain as in \eqref{eq-mproduct}, \eqref{eq-omegaproduct}. If $f\in \mathcal{A}^{-\infty}(\Omega)$, then $\bc f \in \mathcal{Y}^{0,1}_\Omega(\Mm)$.
\end{prop}

\begin{proof}In view of Propositions~\ref{prop-weinstock} and~\ref{prop-bcfacewise}, we only need to 
prove \eqref{eq-canonicality}, which we do in the equivalent form \eqref{eq-canonicality2}. We first note that 
the statement \eqref{eq-canonicality2} is local in the following sense: to prove \eqref{eq-canonicality2} it suffices
to give an open cover $\mathcal{W}$ of $\Mm_k$ and another open cover $\widehat{\mathcal{{W}}}$ of 
$\partial \widehat{D}_k$ by open sets of $\widehat{\Mm}_k$ such that for each $W\in \mathcal{W}$ and 
$\widehat{W}\in\widehat{\mathcal{W}}$, we have
\begin{equation}\label{eq-localcanonicality}
(\id_k\csor \ce_{\widehat{k}})\left( \gamma_k|_{W\times(\widehat{D}_k\cap \widehat{W})}\right)=\gamma_k|_{W\times \widehat{W}},
\end{equation}
where $\id_k$ is now the identity operator on $\Dd'_{0,1}(W)$ and $\ce_{\widehat{k}}$ denotes the canonical extension 
operator from $\mathcal{A}^{-\infty}(\widehat{D}_k\cap \widehat{W}))$ to $\Dd'_0(\widehat{W})$. 
For any point $p\in \Mm_k\times \partial\widehat{D}_k$, we can find a neighborhood $U$ of $p$ in $\Mm$ such that there 
is a vector $v$ as in Theorem~\ref{thm-bcexistence} pointing outward from $\Omega$ along each $\partial\Omega_j$, so that if we define $f_\epsilon=f(\cdot-\epsilon v)$, then \eqref{eq-bc}, \eqref{eq-ce} and \eqref{eq-bcfcomp}  hold. Note 
further that by shrinking $U$ around $p$, we may assume that $U=W\times \widehat{W}$, where $W$ is an open set in $\Mm_k$  and $\widehat{W}$ is an open set of $\widehat{\Mm}_k$. 

Now from \eqref{eq-bcfcomp} we conclude that for each $\phi\in \Dd^{n,n-1}(U)$, we have
\[ \pair{\gamma_k|_U}{\phi}= \lim_{\epsilon\downarrow 0}\pair{\gamma_k^\epsilon}{\phi},\]
where $\gamma_k$ has the same meaning as above, and $\gamma_k^\epsilon=\iota^k_*\left(f_\epsilon[F_k\cap U]\right)^{0,1}$. Let $\{\epsilon_\nu\}$ be a sequence of positive real numbers which converge to the limit 0 as $\nu\to \infty$. Recalling that in $\Dd'_{0,1}(U)$ (or more generally in the dual of a Montel space, see \cite[Section~34.4]{trevesbook}) a weak-* convergent sequence is also  convergent in the usual strong topology, we have 
\begin{equation}\label{eq-gammaepnu}
  \gamma_k|_U = \lim_{\nu\to \infty} \gamma_k^{\epsilon_\nu}.
\end{equation}
Note now that for each $\nu$, we have
\[ \gamma_k^{\epsilon_\nu}|_{W \times (W\cap \widehat{D}_k)}\in \Dd'_{0,1}(W)\csor \mathcal{A}^\infty(\widehat{W}\cap \widehat{D}_k),\]
where $\mathcal{A}^\infty(\widehat{W}\cap \widehat{D}_k)$ denotes the space of functions which are holomorphic 
on $\widehat{W}\cap \widehat{D}_k\subset \widehat{\Mm}_k$ and extend smoothly to the boundary. This follows from the fact that for $\epsilon>0$, we have $\gamma_k^\epsilon = \iota^k_*\left(f_\epsilon[F_k\cap U]\right)^{0,1}$, and $f_\epsilon$ is $\mathcal{C}^\infty$-smooth on $\partial\Omega\cap U$. Therefore we have
\[
(\id_k\csor \ce_{\widehat{k}})\left( \gamma_k^{\epsilon_\nu}|_{W\times(\widehat{D}_k\cap \widehat{W})}\right)=\gamma_k^{\epsilon_\nu}|_{W\times \widehat{W}},
\]
since for functions continuous up to the boundary, the canonical extension is precisely the extension by 0 (see the proof of Theorem~\ref{thm-ce}). We now let $\nu\to\infty$, use \eqref{eq-gammaepnu} and the continuity of $\id_k$ and 
$\ce_{\widehat{k}}$ to conclude that \eqref{eq-localcanonicality} holds.
\end{proof}

\section{Holomorphic extension of currents in $\mathcal{Y}^{0,1}_\Omega(\Mm)$}
\subsection{The structure of the direct summands}\label{sec-gammakstruct} Let $\gamma\in \mathcal{Y}^{0,1}_\Omega(\Mm)$. Using the decomposition \eqref{eq-01dec}, we write $\gamma=\sum_{k=1}^N \gamma_k$,
where $\gamma_k\in \mathcal{D}'_{0,1}(\Mm_k)\csor \mathcal{D}'_{0}(\widehat{\Mm}_k)$. 
First we note the following fact:
\begin{prop} For each $k=1,\dots, N$:
\begin{equation}\label{eq-gammakclaim}
   \gamma_k|_{\Mm_k\times \widehat{D}_k}\in   \mathcal{X}^{0,1}_{D_k}(\Mm_k)\csor \mathcal{A}^{-\infty}(\widehat{D}_k).
\end{equation}
\end{prop}
\begin{proof}Let $n_k=\dim_{\cx}\Mm_k$, and $\hat{n}_k=\dim_{\cx}\widehat{\Mm}_k$. We first show that $\gamma_k\in \mathcal{X}^{0,1}_{D_k}(\Mm_k)\csor \Dd'_0(\widehat{\Mm}_k)$. From the definition of $\mathcal{X}^{0,1}_{D_k}(\Mm_k)$ (see \eqref{eq-weinstock} and \eqref{eq-facedistn}) combined with Lemma~\ref{lem-facestruct}, we see that a current $\theta\in \Dd'_{0,1}(\Mm_k)\csor \Dd'_0(\widehat{\Mm}_k)$ 
is  in  $\mathcal{X}^{0,1}_{D_k}(\Mm_k)\csor \Dd'_0(\widehat{\Mm}_k)$ provided the following two conditions are satisfied:
 \begin{enumerate}
 \item[(i)] for each $\omega\in \mathcal{D}^{n_k,n_k-1}(\Mm_k)$ such that $\dbar \omega=0$ on $D_k$,
 and for each  $\phi\in \Dd^{\hat{n}_k,\hat{n}_k}(\widehat{\Mm}_k)$ we have $\pair{\theta}{ \omega\tensor\phi}=0$;
\item[(ii)] denoting as usual by $\jmath^k$ the inclusion of $\partial D_k$ in $\Mm_k$, for each $\lambda\in  \mathcal{D}^{n_k,n_k-1}(\Mm_k)$ such that $(\jmath^k)^*(\lambda)=0$, and for each $\phi\in  \Dd^{\hat{n}_k,\hat{n}_k}(\widehat{\Mm}_k)$ we have $\pair{\theta}{\lambda\tensor \phi}=0$.
\end{enumerate}

Note that in the algebraic tensor products $\omega\tensor \phi$ and $\lambda\tensor\phi$ above, 
the factors have to be reordered.
To verify (i),  note that if $j\not=k$, we have $\pair{\gamma_j}{\omega\tensor\phi}=0$,  since $\omega\tensor \phi \in \Dd^{n_k, n_k-1}(\Mm_k)\csor\Dd^{\hat{n}_k,\hat{n}_k}(\widehat{\Mm}_k)$, whereas $\gamma_j\in \Dd'_{0,1}(\Mm_j)\csor \Dd'_{0,0}(\widehat{\Mm}_j)$. If we set
$\hat{\gamma}_k= \sum_{j\not = k} \gamma_j$, then we have $\pair{\hat{\gamma}_k}{\omega\tensor \phi}=0$.  Now since $\dbar \omega=0$ on  $D_k$,
we see that $\dbar(\omega\tensor \phi)$ vanishes on $D_k\times \widehat{\Mm}_k$, and therefore vanishes 
a fortiori on $\Omega=D_k\times \widehat{D}_k$. Since
$\gamma\in \mathcal{X}^{0,1}_\Omega(\Mm)$, we have therefore $\pair{\gamma}{\omega\tensor\phi}=0$. We therefore have
\begin{align*}\pair{\gamma_k}{\omega\tensor\phi}&= \pair{\gamma-\hat{\gamma}_k}{\omega\tensor\phi}\\
&=\pair{\gamma}{\omega\tensor\phi}- \pair{\hat{\gamma}_k}{\omega\tensor\phi}\\
&=0.
\end{align*}

For (ii), we use the representation \eqref{eq-gammakfacewise} of $\gamma_k$. We then have
\begin{align*}
\pair{\gamma_k}{\lambda\tensor \phi}&= \pair{(\jmath^k_*\csor \id)(\alpha_k)^{0,1}}{\lambda\tensor\phi}\\
&=\pair{\alpha_k}{(\jmath^k)^*\csor \id)(\lambda\tensor\phi)}\\
&= \pair{\alpha_k}{(\jmath^k)^*\lambda\tensor\phi)}\\
&=\pair{\alpha_k}{0}\\
&=0.
\end{align*}
Therefore, it follows that $\gamma_k\in \mathcal{X}^{0,1}_{D_k}(\Mm_k)\csor \Dd'_0(\widehat{\Mm}_k)$.  Now, 
we can write $\gamma_k|_{\Mm_k\times \widehat{D}_k}= \iota^k_*(\alpha_k|_{\partial D_k\times \widehat{D}_k})^{0,1}= (\jmath^k_*\csor \id)\left(\alpha_k|_{\partial D_k\times \widehat{D}_k}\right)^{0,1}$.
Using \eqref{eq-alphajstruct}, the  result \eqref{eq-gammakclaim} now follows.
   \end{proof}
\subsection{The space $\widetilde{\mathcal{A}}^{-\infty}(\Omega)$}   We begin by noting some simple properties of the space of holomorphic functions of polynomial growth:
   \begin{prop}\label{prop-ainftyproperties} Let  $D$ be a domain with generic corners  in the complex manifold $\Mm$. Then
\begin{enumerate}
\item[(a)] the canonical  extension map 
\[\ce: \mathcal{A}^{-\infty}(D)\rightarrow \Dd'_0(\Mm)\]
is an isomorphism (of TVS) onto the image  $\ce(\mathcal{A}^{-\infty}(D))$, equipped with the subspace topology from $\Dd'_0(\Mm)$.
\item[(b)] The space $\mathcal{A}^{-\infty}(D)$ is nuclear. Consequently, there is a naturally defined topological  tensor 
product $\mathcal{A}^{-\infty}(D)\csor \Xx$ with any locally convex topological vector space $\Xx$, which can be naturally identified with a closed subspace of $\Dd'_0(\Mm, \Xx)$, the space of $\Xx$-valued distributions.
\item[(c)] if $U$ is a relatively compact open subset in a complex manifold $\Nn$, then
\[ \mathcal{A}^{-\infty}(D)\csor \mathcal{A}^{-\infty}(U)\subseteq \mathcal{A}^{-\infty}(D\times U).\]
\end{enumerate}
\end{prop}

\begin{proof} Since $\ce$ is obviously injective. to prove (a), it suffices to show that $\ce(\mathcal{A}^{-\infty}(D))$ with its subspace topology is a DFS space. Then the result would follow from Prop~\ref{res-dfsinverse}. We claim that each element $\ce f$ of $\ce(\mathcal{A}^{-\infty}(D))$ induces a linear functional on $\mathcal{E}^{n,n}(\overline{D})$, the space of 
top degree forms smooth up to the boundary on $D$. If $\phi\in \mathcal{E}^{n,n}(\overline{D})$, and $\tilde{\phi}$ 
is any extension of $\phi$ to an element of $\Dd^{n,n}(\Mm)$, then we define $\pair{\ce f}{\phi}=\pair{\ce f}{\tilde{\phi}}$, which is well-defined, since from the definition it is clear that $\ce f$ vanishes on any test form vanishing on $D$. This embeds $\ce(\mathcal{A}^{-\infty}(D))$ as a closed subspace of the strong dual of $\mathcal{E}^{n,n}(\overline{D})$. Since $\mathcal{E}^{n,n}(\overline{D})$ is a Fr\'{e}chet-Schwartz space, it follows that $\ce(\mathcal{A}^{-\infty}(D))$ is a DFS space. 

Assertion (b) now follows, since $\ce$ embeds $\mathcal{A}^{-\infty}(D)$ as a closed subspace of the nuclear space $\Dd'_0(\Mm)$.
This also allows us to view $\mathcal{A}^{-\infty}(D)\csor \Xx$ as the subspace $\ce(\mathcal{A}^{-\infty}(D))\csor\Xx$
of $\Dd'_0(\Mm)\csor \Xx$.

To see (c), note that it suffices to prove that the algebraic tensor product $\mathcal{A}^{-\infty}(D)\tensor \mathcal{A}^{-\infty}(U)$ is contained in $\mathcal{A}^{-\infty}(D\times U)$. Indeed, by linearity, it suffices to show that 
if $f\in \mathcal{A}^{-\infty}(D)$ and $g\in \mathcal{A}^{-\infty}(U)$, then $f\tensor g \in \mathcal{A}^{-\infty}(D\times U)$. Now $\mathcal{A}^{-\infty}(D\times U)$ is closed under multiplication (since on any domain $\Omega$, we clearly have for $F\in \mathcal{A}^{-k}(\Omega), G\in \mathcal{A}^{-\ell}(\Omega)$ that $FG\in \mathcal{A}^{-(k+\ell)}(\Omega)$.) Denoting by ${\bf 1}$ the constant function with value 1, we see that $f\tensor {\bf 1}$ and ${\bf 1}\tensor g$ belong to $\mathcal{A}^{-\infty}(D\times U)$, but we have $f\tensor g= (f\tensor {\bf 1})\cdot ({\bf 1}\tensor g)$. 
\end{proof}

\subsection{Proof of Theorem~\ref{thm-product}} Let $\Omega=D_1\times\dots\times D_N$ be a product domain in a product manifold as in \eqref{eq-omegaproduct}. We define:
   \[ \widetilde{\mathcal{A}}^{-\infty}(\Omega)= \mathcal{A}^{-\infty}(D_1)\csor \mathcal{A}^{-\infty}(D_2)\csor \dots
 \csor \mathcal{A}^{-\infty}(D_N),\]
 where the topological tensor products are well-defined thanks to Proposition~\ref{prop-ainftyproperties}. Then
 $ \widetilde{\mathcal{A}}^{-\infty}(\Omega)\subset  {\mathcal{A}}^{-\infty}(\Omega)$. Similarly, we denote
 by $ \widetilde{\mathcal{A}}^{-\infty}(\widehat{D}_k)$ the topological tensor product of  the $\mathcal{A}^{-\infty}(D_j)$'s with $j\not =k$. We will need the following properties of $\widetilde{\mathcal{A}}^{-\infty}(\Omega)$:
 
 \begin{lem}\label{lem-bcainfty}
   The map $\ce: \mathcal{A}^{-\infty}(\Omega)\to \Dd'_0(\Mm)$ 
 restricted to the subspace $\widetilde{\mathcal{A}}^{-\infty}(\Omega)$ admits the representation
 \begin{equation}\label{eq-cetildea}
 \ce|_{ \widetilde{\mathcal{A}}^{-\infty}(\Omega)}= \widehat{\otimes}_{k=1}^N \ce_k,
 \end{equation}
 where $\ce_k:\mathcal{A}^{-\infty}(D_k)\to \Dd'_0(\Mm_k)$ is the canonical extension map.
   Similarly, the restriction of 
 $\bc:  \mathcal{A}^{-\infty}(\Omega)\to \Dd'_{0,1}(\Mm)$ when  admits the representation 
 \begin{equation}\label{eq-bctildea}
   \bc|_{\widetilde{ \mathcal{A}}^{-\infty}(\Omega)} = \sum_{k=1}^N \bc_k \csor \ce_{\widehat{k}},
\end{equation}
 where $\bc_k: \mathcal{A}^{-\infty}(D_k)\to \Dd'_{0,1}(\Mm_k)$ is the boundary current map and $\ce_{\widehat{k}}:\widetilde{\mathcal{A}}^{-\infty}(\widehat{D}_k)\to \Dd'_0(\widehat{\Mm}_k)$ is restriction to $\widetilde{\mathcal{A}}^{-\infty}(\widehat{D}_k)$ of the canonical extension map.
 \end{lem}
 
 \begin{proof} In order to establish \eqref{eq-cetildea} it suffices to show that 
 \begin{equation}\label{eq-cetensor}
\ce f = \ce_1 f_1\tensor \ce_2 f_2 \tensor \dots \tensor\ce_N f_N,
\end{equation}
 whenever $f_j\in \mathcal{A}^{-\infty}(\Omega_j)$ and  $f=f_1\tensor\dots \tensor f_N$ is their tensor product, which lies in $\mathcal{A}^{-\infty}(\Omega)$ by part (c) of Proposition~\ref{prop-ainftyproperties}. Once \eqref{eq-cetensor} is established, it follows by linearity that \eqref{eq-cetildea} holds on the algebraic tensor product
$\mathcal{A}^{-\infty}(D_1)\tensor\dots \tensor\mathcal{A}^{-\infty}(D_N)$, and then \eqref{eq-cetildea} follows by density.

Note that \eqref{eq-cetildea} is a local property, in the sense that to prove it, it suffices to show that each point in  the product manifold $\Mm=\Mm_1\times\dots \Mm_n$ has a neighborhood $W$ of the form $W_1\times\dots\times W_N$, with $W_j\subset \Mm_j$,  such that we have
 \begin{equation}\label{eq-cetensorlocal}
 \ce (f|_{\Omega\cap W}) = \ce_1 (f_1|_{D_1\cap W_1})\tensor  \dots \tensor\ce_N (f_N|_{D_N\cap W_N}),
 \end{equation}
where $\ce_j$ is now the canonical extension operator from $\mathcal{A}^{-\infty}(W_j\cap D_j)$ into $\Dd'_0(W_j)$, 
and $\ce$ is the  canonical extension operator from $\mathcal{A}^{-\infty}(W\cap D)$ into $\Dd'_0(W)$.
To prove \eqref{eq-cetensorlocal}, we use the method of proof in Proposition~\ref{prop-y01target}, i.e., represent  $f$ and $f_j$ locally as a limit of functions smooth up to the boundary. Note that in \eqref{eq-cetensorlocal}, the only interesting case is when $W$ is a neighborhood of a point on the boundary of $\Omega$. After shrinking $W$, we may assume that that there is a vector $v$ as in Theorem~\ref{thm-ce} which is transverse to each $\partial\Omega_j$ in $W$, and let 
$f_\epsilon= f(\cdot-\epsilon v)$ be also as in Theorem~\ref{thm-ce}. Choosing a sequence of positive numbers $\epsilon_\nu$ such that 
$\epsilon_\nu\to 0$ as $\nu\to \infty$, we see that $f_{\epsilon_\nu}$ is smooth up to the boundary on $\Omega\cap W$, and we have
\[ f_{\epsilon_\nu}= f_1^\nu\tensor\dots \tensor f_N^\nu,\]
where $f_j^\nu\in \mathcal{O}(W_j\cap D_j)$ is smooth up to the boundary as well, and as in the proof of Proposition~\ref{prop-y01target}, as $\nu\to \infty$, we have $f_j^\nu\to f_j$ and $f_{\epsilon_\nu}\to f$.
Since \eqref{eq-cetensorlocal} holds when $f$ is replaced by $f_{\epsilon_nu}$ and $f_j$ is replaced by $f_j^\nu$, and since  the maps $\ce_j$ and $\ce$ are continuous, taking a limit as $\nu\to \infty$, \eqref{eq-cetensorlocal} follows, and from which follows \eqref{eq-cetildea}.

We now prove \eqref{eq-bctildea} by a direct computation using the definition \eqref{eq-bcdef}. We obtain
\begin{align*}
 \bc|_{\widetilde{ \mathcal{A}}^{-\infty}(\Omega)}&= -\dbar \left(  \ce|_{\widetilde{ \mathcal{A}}^{-\infty}(\Omega)}\right)\\
 &= -\dbar \left( \widehat{\otimes}_{k=1}^N \ce_k,\right)\\
 &= \sum_{k=1}^N  -(\dbar \circ\ce_k)\csor \ce_{\widehat{k}} \\
&=  \sum_{k=1}^N \bc_k \csor \ce_{\widehat{k}},
\end{align*}
where in the last-but-one line we have used the Leibniz product rule, and the fact that $\ce_{\widehat{k}}$ on $\widetilde{\mathcal{A}}^{-\infty}(\widehat{D}_k)$ is by \eqref{eq-cetildea} the tensor product $\widehat{\otimes}_{j\not=k} \ce_j$.
\end{proof}
 We now prove the following result, which contains Theorem~\ref{thm-product}:

\begin{prop}Let $\Mm,\Omega$ be as in \eqref{eq-mproduct} and \eqref{eq-omegaproduct}. 
Then we have the following:
\begin{enumerate}
\item $\bc:\mathcal{A}^{-\infty}(\Omega)\to \mathcal{Y}^{0,1}_\Omega(\Mm)$ is an isomorphism of topological
vector spaces, and
\item $\widetilde{\mathcal{A}}^{-\infty}(\Omega)= \mathcal{A}^{-\infty}(\Omega)$
\end{enumerate}
\end{prop}
\begin{proof}We proceed by induction on $N$, the number of smooth factors of the product $\Mm$. For $N=1$, conclusion (1) is Theorem~\ref{thm-smoothcase} and conclusion (2) is obvious. Therefore we assume the result when 
$\Omega$ has $N-1$ smooth factors. Note that each $k=1,\dots, N$, we have
\[ \widetilde{\mathcal{A}}^{-\infty}(\widehat{D}_k)= \mathcal{A}^{-\infty}(\widehat{D}_k),\]
since the domain $\widehat{D}_k\subset \widehat{\Mm}_k$ is the product of $(N-1)$ smooth factors. 

We now show that $\bc:\mathcal{A}^{-\infty}(\Omega)\to \mathcal{Y}^{0,1}_\Omega(\Mm)$ is an isomorphism of 
topological vector spaces. The injectivity follows from Proposition~\ref{prop-bcinjective}, so we need only to show that 
$\bc$ is surjective. By Theorem~\ref{thm-smoothcase}, the map $\bc: \mathcal{A}^{-\infty}(D_k)\to \mathcal{X}^{0,1}_{D_k}(\Mm_k)$
 is an isomorphism of 
topological vector spaces. Denote by $\bc_k^{-1}$ its inverse, which is an isomorphism from $\mathcal{X}^{0,1}_{D_k}(\Mm_k)$ to $\mathcal{A}^{-\infty}(D_k)$. Tensoring with the identity map on $\widetilde{\mathcal{A}}^{-\infty}(\widehat{D}_k)$, we obtain an isomorphism 
\[ (\bc_k^{-1}\csor \id): \mathcal{X}^{0,1}_{D_k}(\Mm_k)\csor \mathcal{A}^{-\infty}(\widehat{D}_k)\to \widetilde{\mathcal{A}}^{-\infty}(\Omega),\]
where we use the fact that $\widetilde{\mathcal{A}}^{-\infty}(\Omega)={\mathcal{A}}^{-\infty}(D_k)\csor \widetilde{\mathcal{A}}^{-\infty}(\widehat{D}_k)$.

Now let $\gamma\in \mathcal{Y}^{0,1}_\Omega(\Mm)$, and let $\gamma_k\in \Dd'_{0,1}(\Mm_k)\csor \Dd'_{0,0}(\widehat{\Mm}_k)$, be as in the beginning of Section~\ref{sec-gammakstruct}.  Then by \eqref{eq-gammakclaim} and the induction hypothesis, we have that $\gamma_k|_{\Mm_k\times \widehat{D}_k} \in \mathcal{X}^{0,1}_{D_k}(\Mm_k)\csor \widetilde{\mathcal{A}}^{-\infty}(\widehat{D}_k).$ We define $f_k\in \widetilde{\mathcal{A}}^{-\infty}(\Omega)$  as 
\[ f_k = (\bc_k^{-1}\csor \id)\left( \gamma_k|_{\Mm_k\times \widehat{D}_k}\right),\]
which makes sense thanks to \eqref{eq-gammakclaim}. The proof of the surjectivity of $\bc$ will be completed by
showing that $\bc f_k= \gamma$.

Set $\lambda= \bc f_k$, and let $\lambda= \sum_{j=1}^N \lambda_j$ be the standard decomposition of $\lambda$ as 
in \eqref{eq-01dec2}. We claim that $\lambda_k=\gamma_k$. Indeed, in the representation \eqref{eq-bctildea}, the summands correspond to the terms of the standard decomposition \eqref{eq-01dec2}. Therefore, we have
\begin{align}\lambda_k &=(\bc_k\csor \ce_{\widehat{k}})f_k\nonumber\\
&=(\bc_k\csor \ce_{\widehat{k}}) (\bc_k^{-1}\csor \id)\left( \gamma_k|_{\Mm_k\times \widehat{D}_k}\right)\nonumber\\
&= (\id_k \csor \ce_{\widehat{k}})\left( \gamma_k|_{\Mm_k\times \widehat{D}_k}\right)\nonumber\\
&=\gamma_k,\label{eq-lambdakgammak}
\end{align}
thanks to the canonicality of the face distributions, as expressed in \eqref{eq-canonicality2}. To complete the proof
we will show that for each pair $k, \ell\in \{1,\dots, N\}$, we have $f_k=f_\ell$, which will ensure that $\bc f_k=\gamma$.

Consider the continuous linear map from $\widetilde{\mathcal{A}}^{-\infty}(\Omega)$ into $\Dd'_{0,2}(\Mm)$ given by
\[ \mathsf{b}_{k,\ell}= \ce \csor\dots \csor \bc \csor \dots \csor \bc \csor \dots \csor \ce,\]
where each tensor factor is $\ce$ except the $k$-th and $\ell$-th ones, which are $\bc$. Then $\mathsf{b}_{k,\ell}:\widetilde{\mathcal{A}}^{-\infty}(\Omega)\to\Dd'_{0,2}(\Mm)$ is injective, since each of the factor maps $\ce$ of $\bc$ of the tensor product 
defining $\mathsf{b}_{k,\ell}$ is injective. Using the fact that $\bc =-\dbar \circ \ce$, we obtain the representation
\begin{equation}\label{eq-bkl1}
\mathsf{b}_{k,\ell}= - \dbar_{\Mm_\ell}\circ (\bc_k \csor \ce_{\widehat{k}}),
\end{equation}
where $\dbar_{\Mm_\ell}$ is  the differential operator as in \eqref{eq-dbarmk}, the tensor product of the $\dbar$ operator on the factor $\Mm_\ell$ and the identity operator on $\widehat{\Mm}_\ell$, and $\bc_k \csor \ce_{\widehat{k}}$ is as in \eqref{eq-bctildea}. 
Of course, by symmetry we may also write
\begin{equation}\label{eq-bkl2}
  \mathsf{b}_{k,\ell}= - \dbar_{\Mm_k}\circ (\bc_\ell \csor \ce_{\widehat{\ell}}).
\end{equation}
Therefore, we have 
\begin{align*}  \mathsf{b}_{k,\ell} f_k & =  - \dbar_{\Mm_{\ell}}\circ (\bc_k \csor \ce_{\widehat{k}})f_k\\
&=  - \dbar_{\Mm_{\ell}}\gamma_k,
\end{align*}
and also 
\begin{align*}  \mathsf{b}_{k,\ell} f_\ell & =  - \dbar_{\Mm_{k}}\circ (\bc_\ell \csor \ce_{\widehat{\ell}})f_\ell\\
&=  - \dbar_{\Mm_{k}}\gamma_\ell.
\end{align*}
Using \eqref{eq-edge} of Proposition~\ref{prop-proddbarclosed} we see that $\dbar_{\Mm_{\ell}}\gamma_k=\dbar_{\Mm_{k}}\gamma_\ell$, therefore $\mathsf{b}_{k,\ell} f_k=\mathsf{b}_{k,\ell} f_\ell$, so that by injectivity of $\mathsf{b}_{k,\ell}$ it follows that $f_k$ and $f_\ell$ are the same function in $\widetilde{\mathcal{A}}^{-\infty}(\Omega)$.  This completes the proof of the surjectivity of 
$\bc:\mathcal{A}^{-\infty}(\Omega)\to \mathcal{Y}^{0,1}_\Omega(\Mm)$, and therefore it is a bijection, since 
we already know that $\bc$ is injective.

Note further that the inverse mapping to $\bc$ constructed during the above argument actually maps into $\widetilde{\mathcal{A}}^{-\infty}(\Omega)$. It follows that $\widetilde{\mathcal{A}}^{-\infty}(\Omega)={\mathcal{A}}^{-\infty}(\Omega)$.

Finally, both $\mathcal{A}^{-\infty}(\Omega)$ and $\mathcal{Y}^{0,1}_\Omega(\Mm)$ are DFS space, the latter being a closed subspace of $\Dd'_{0,1}(\Mm)$. It follows therefore from 
Proposition~\ref{res-dfsinverse} that $\bc$ is, in fact, 
an isomorphism of topological vector spaces. This completes the induction, and the proposition (and Theorem~\ref{thm-product})  is proved.
\end{proof}
\subsection{Boundary value on the distinguished boundary} For the product domain $\Omega\Subset\Mm$ of \eqref{eq-omegaproduct}, one can also consider boundary values on the {\em distinguished} or {\em \v{S}ilov} boundary
\[ \partial_\Sha \Omega= \partial D_1\times\dots\times \partial D_N,\]
which is a smooth submanifold of $\Mm$ of codimension $N$. In this section we give a brief account of such boundary values, omitting the routine proofs. We consider the operator $\bc_\Sha$ from $\mathcal{A}^{-\infty}(\Omega)=\widetilde{\mathcal{A}}^{-\infty}(\Omega)$ to $\Dd'_{0,N}(\Mm)$ (currents of 
bidegree $(0,N)$ on $\Mm$) given by
\[ \bc_\Sha = \widehat{\bigotimes}_{k=1}^N \bc_k=\bc_1\csor\dots\csor\bc_N,\]
where on the right hand side, $\bc_k:\mathcal{A}^{-\infty}(D_k)\to \mathcal{X}_{D_k}^{0,1}(\Mm)\subset\Dd'_{0,1}(\Mm)$ is the boundary current operator on the smooth domain $D_k$. Since for each $k$, the map $\bc_k:\mathcal{A}^{-\infty}(D_k)\to \mathcal{X}_{D_k}^{0,1}(\Mm)$ is an isomorphism of topological vector spaces, 
we conclude that we have an isomorphism
\[ \bc_\Sha : \mathcal{A}^{-\infty}(\Omega)\to \mathcal{X}^{0,N}_\Sha(\Mm),\]
where 
\[ \mathcal{X}^{0,N}_\Sha(\Mm) =\mathcal{X}^{0,1}_{D_1}(\Mm_1)\csor \dots \csor\mathcal{X}^{0,1}_{D_N}(\Mm_N)\subset \Dd'_{0,N}(\Mm).\]
Using the definition of the spaces $\mathcal{X}^{0,1}_{D_k}(\Mm_k)$ in terms of the Weinstock condition \eqref{eq-weinstock} and the existence of a face distribution \eqref{eq-facedistn}, we have the following easy consequence:
\begin{prop} Let $\Gamma\in \Dd'_{0,N}(\Mm)$. Then there is a  holomorphic $f\in \mathcal{A}^{-\infty}(\Omega)$ 
such that $\Gamma= \bc_\Sha f$ if and only if the following two conditions hold
\begin{enumerate}\item Let $\omega\in \Dd^{n,n-N}(\Mm)$ be such that for each $k=1,\dots, N$, we have $\dbar_{\Mm_k}\omega=0$ on $\overline{\Omega}$, where $\dbar_{\Mm_k}$ is as in \eqref{eq-dbarmk}.
Then we have $\pair{\Gamma}{\omega}=0$.
 \item  There is a distribution $A\in \Dd'_0(\Gamma)$ such that $\Gamma=I_*(A)^{(0,N)}$, where $I:\partial_\Sha \Omega\to \Mm$ is the inclusion map, and the superscript $(0,N)$ denotes taking the part of bidegree $(0,N)$ of 
 the $N$-current $I_*(A)$.
\end{enumerate}
\end{prop} 

Further, from \eqref{eq-bc}, we can obtain the following local representation of the boundary value on the \v{S}ilov boundary. 
\begin{prop}Let $p\in \partial_\Sha\Omega$, let $U$ be a coordinate chart of $\Mm$ around $p$, $U=U_1\times \dots\times U_N$, where $U_k\subset \Mm_k$, and let $v$ be a vector such that its projection on each $\Mm_k$ is transverse to $\partial D_k$ inside $U_k$. Then, for each $\psi\in \Dd^{n, n-N}(U)$, we have
\[ \pair{\bc_\Sha f}{\psi} = \lim_{\epsilon\downarrow 0} \int_{\partial_\Sha\Omega} f_\epsilon \psi,\]
where $f_\epsilon(z)= f(z-\epsilon v)$.
\end{prop}
\bibliographystyle{alpha}
\bibliography{hartogs}

\begin{thebibliography}{BER99}

\bibitem[Bar95]{barrettduality}
David~E. Barrett.
\newblock Duality between {$A^\infty$} and {$A^{-\infty}$} on domains with
  nondegenerate corners.
\newblock In {\em Multivariable operator theory ({S}eattle, {WA}, 1993)},
  volume 185 of {\em Contemp. Math.}, pages 77--87. Amer. Math. Soc.,
  Providence, RI, 1995.

\bibitem[Bel82]{bell_harm}
Steven~R. Bell.
\newblock A duality theorem for harmonic functions.
\newblock {\em Michigan Math. J.}, 29(1):123--128, 1982.

\bibitem[BER99]{baouendibook}
M.~Salah Baouendi, Peter Ebenfelt, and Linda~Preiss Rothschild.
\newblock {\em Real submanifolds in complex space and their mappings},
  volume~47 of {\em Princeton Mathematical Series}.
\newblock Princeton University Press, Princeton, NJ, 1999.

\bibitem[CT94]{cordarotreves}
Paulo~D. Cordaro and Fran{\c{c}}ois Tr{{\`e}}ves.
\newblock {\em Hyperfunctions on hypo-analytic manifolds}, volume 136 of {\em
  Annals of Mathematics Studies}.
\newblock Princeton University Press, Princeton, NJ, 1994.

\bibitem[CV13]{chak-kaushal2}
Debraj Chakrabarti and Kaushal Verma.
\newblock Condition {R} and holomorphic mappings of domains with generic
  corners.
\newblock {\em Illinois J. Math.}, 57(4):1035--1055, 2013.

\bibitem[For93]{forst93}
Franc Forstneri{\v{c}}.
\newblock A reflection principle on strongly pseudoconvex domains with generic
  corners.
\newblock {\em Math. Z.}, 213(1):49--64, 1993.

\bibitem[HL75]{hala1}
F.~Reese Harvey and H.~Blaine Lawson, Jr.
\newblock On boundaries of complex analytic varieties. {I}.
\newblock {\em Ann. of Math. (2)}, 102(2):223--290, 1975.

\bibitem[Hor91]{horvath-review}
John Horv{{\'a}}th.
\newblock Book {R}eview: {D}istributions and analytic functions.
\newblock {\em Bull. Amer. Math. Soc. (N.S.)}, 25(1):162--170, 1991.

\bibitem[KR65]{kohnrossi}
J.~J. Kohn and Hugo Rossi.
\newblock On the extension of holomorphic functions from the boundary of a
  complex manifold.
\newblock {\em Ann. of Math. (2)}, 81:451--472, 1965.

\bibitem[Kyt95]{kytbook}
Alexander~M. Kytmanov.
\newblock {\em The {B}ochner-{M}artinelli integral and its applications}.
\newblock Birkh{\"a}user Verlag, Basel, 1995.
\newblock Translated from the Russian by Harold P. Boas and revised by the
  author.

\bibitem[Lan80]{landucci}
Mario Landucci.
\newblock Boundary values of holomorphic functions and {C}auchy problem for
  {$\bar \partial $} operator in the polydisc.
\newblock {\em Rend. Sem. Mat. Univ. Padova}, 62:23--34, 1980.

\bibitem[Mal57]{mal57}
Bernard Malgrange.
\newblock Faisceaux sur des vari{\'e}t{\'e}s analytiques r{\'e}elles.
\newblock {\em Bull. Soc. Math. France}, 85:231--237, 1957.

\bibitem[Mar64]{martineau1}
A.~Martineau.
\newblock Distributions et valeurs au bord des fonctions holomorphes.
\newblock In {\em Theory of {D}istributions ({P}roc. {I}nternat. {S}ummer
  {I}nst., {L}isbon, 1964)}, pages 193--326. Inst. Gulbenkian Ci., Lisbon,
  1964.

\bibitem[Mor93]{morimoto}
Mitsuo Morimoto.
\newblock {\em An introduction to {S}ato's hyperfunctions}, volume 129 of {\em
  Translations of Mathematical Monographs}.
\newblock American Mathematical Society, Providence, RI, 1993.
\newblock Translated and revised from the 1976 Japanese original by the author.

\bibitem[MV97]{vogt}
Reinhold Meise and Dietmar Vogt.
\newblock {\em Introduction to functional analysis}, volume~2 of {\em Oxford
  Graduate Texts in Mathematics}.
\newblock The Clarendon Press, Oxford University Press, New York, 1997.
\newblock Translated from the German by M. S. Ramanujan and revised by the
  authors.

\bibitem[Nar95]{narasimhan}
Raghavan Narasimhan.
\newblock {\em Several complex variables}.
\newblock Chicago Lectures in Mathematics. University of Chicago Press,
  Chicago, IL, 1995.
\newblock Reprint of the 1971 original.

\bibitem[PW78]{polkingwells}
John~C. Polking and R.~O. Wells, Jr.
\newblock Boundary values of {D}olbeault cohomology classes and a generalized
  {B}ochner-{H}artogs theorem.
\newblock {\em Abh. Math. Sem. Univ. Hamburg}, 47:3--24, 1978.
\newblock Special issue dedicated to the seventieth birthday of Erich
  K{{\"a}}hler.

\bibitem[Ran02]{range}
R.~Michael Range.
\newblock Extension phenomena in multidimensional complex analysis: correction
  of the historical record.
\newblock {\em Math. Intelligencer}, 24(2):4--12, 2002.

\bibitem[Ser55]{serre}
Jean-Pierre Serre.
\newblock Un th{\'e}or{\`e}me de dualit{\'e}.
\newblock {\em Comment. Math. Helv.}, 29:9--26, 1955.

\bibitem[Str84]{straube84}
Emil~J. Straube.
\newblock Harmonic and analytic functions admitting a distribution boundary
  value.
\newblock {\em Ann. Scuola Norm. Sup. Pisa Cl. Sci. (4)}, 11(4):559--591, 1984.

\bibitem[Tr{\`{e}}67]{trevesbook}
Fran{\c{c}}ois Tr{\`{e}}ves.
\newblock {\em Topological vector spaces, distributions and kernels}.
\newblock Academic Press, New York-London, 1967.

\bibitem[Web82]{webster82}
S.~M. Webster.
\newblock Holomorphic mappings of domains with generic corners.
\newblock {\em Proc. Amer. Math. Soc.}, 86(2):236--240, 1982.

\bibitem[Wei69]{weinstock1}
Barnet~M. Weinstock.
\newblock Continuous boundary values of analytic functions of several complex
  variables.
\newblock {\em Proc. Amer. Math. Soc.}, 21:463--466, 1969.

\bibitem[Wei70]{weinstock2}
Barnet~M. Weinstock.
\newblock An approximation theorem for {$\overline \partial $}-closed forms of
  type {$(n,\,n-1)$}.
\newblock {\em Proc. amer. Math. Soc.}, 26:625--628, 1970.

\end{thebibliography}
\end{document}